\documentclass[reqno,11pt]{amsart}

\usepackage[a4paper,left=25mm,right=25mm,top=30mm,bottom=30mm,marginpar=25mm]{geometry} 
\usepackage{amsmath}
\usepackage{amssymb}
\usepackage{amsthm}
\usepackage{amscd}
\usepackage[ansinew]{inputenc}
\usepackage{color}
\usepackage[final]{graphicx}
\usepackage{tikz}
\usepackage{bbm}
\usepackage{changes}

\usetikzlibrary{patterns} 
\usetikzlibrary{positioning}
\usetikzlibrary{calc}
\usetikzlibrary{arrows,shapes,backgrounds}
\usetikzlibrary{intersections} 

\renewcommand{\epsilon}{\varepsilon}

\numberwithin{equation}{section}

\newtheoremstyle{thmlemcorr}{10pt}{10pt}{\itshape}{}{\bfseries}{.}{10pt}{{\thmname{#1}\thmnumber{ #2}\thmnote{ (#3)}}}
\newtheoremstyle{thmlemcorr*}{10pt}{10pt}{\itshape}{}{\bfseries}{.}\newline{{\thmname{#1}\thmnumber{ #2}\thmnote{ (#3)}}}
\newtheoremstyle{defi}{10pt}{10pt}{\itshape}{}{\bfseries}{.}{10pt}{{\thmname{#1}\thmnumber{ #2}\thmnote{ (#3)}}}
\newtheoremstyle{remexample}{10pt}{10pt}{}{}{\bfseries}{.}{10pt}{{\thmname{#1}\thmnumber{ #2}\thmnote{ (#3)}}}
\newtheoremstyle{ass}{10pt}{10pt}{}{}{\bfseries}{.}{10pt}{{\thmname{#1}\thmnumber{ A#2}\thmnote{ (#3)}}}

\theoremstyle{thmlemcorr}
\newtheorem{theorem}{Theorem}
\numberwithin{theorem}{section}
\newtheorem{lemma}[theorem]{Lemma}

\newtheorem{proposition}[theorem]{Proposition}

\theoremstyle{thmlemcorr*}
\newtheorem{theorem*}{Theorem}
\newtheorem{lemma*}[theorem]{Lemma}
\newtheorem{corollary*}[theorem]{Corollary}
\newtheorem{proposition*}[theorem]{Proposition}
\newtheorem{problem*}[theorem]{Problem}
\newtheorem{conjecture*}[theorem]{Conjecture}

\theoremstyle{defi}
\newtheorem{definition}[theorem]{Definition}

\newtheorem{conclusion}[theorem]{Conclusion}

\theoremstyle{remexample}
\newtheorem{remarkx}[theorem]{Remark}
\newenvironment{remark}
  {\pushQED{\qed}\remarkx}
  {\popQED\endremarkx}
\newtheorem{examplex}[theorem]{Example}
\newenvironment{example}
  {\pushQED{\qed}\examplex}
  {\popQED\endexamplex}

\theoremstyle{ass}

\newcommand{\Acal}{\mathcal{A}}

\newcommand{\Fcal}{\mathcal{F}}

\newcommand{\Ical}{\mathcal{I}}

\newcommand{\Lcal}{\mathcal{L}}
\newcommand{\Mcal}{\mathcal{M}}

\newcommand{\Pcal}{\mathcal{P}}

\newcommand{\Vcal}{\mathcal{V}}

\newcommand{\Ycal}{\mathcal{Y}}

\newcommand{\Qbb}{\mathbb{Q}}

\DeclareMathOperator{\esssup}{ess\,sup}

\DeclareMathOperator{\dist}{dist}

\DeclareMathOperator{\supp}{supp}

\newcommand{\norm}[1]{\|#1\|}

\newcommand{\N}{\mathbb{N}}
\newcommand{\R}{\mathbb{R}}

\newcommand{\weakly}{\rightharpoonup}
\newcommand{\weaklystar}{\overset{*}\rightharpoonup}

\newcommand{\eps}{\epsilon}

\def\XXint#1#2#3{{\setbox0=\hbox{$#1{#2#3}{\int}$} 
\vcenter{\hbox{$#2#3$}}\kern-.5\wd0}}

\definecolor{vg}{rgb}{0.0, 0.35, 0.25}

\title[Cartesian convexity in the theory of nonlocal supremal functionals]{ Cartesian convexity as the key notion in the variational existence theory for nonlocal supremal functionals}

\author{Carolin Kreisbeck}
\address{Mathematisch-Geographische Fakult\"at, Katholische Universit\"at Eichst\"att-Ingolstadt, 85071 Eichst\"att, Germany}
\email{carolin.kreisbeck@ku.de}

\author{Antonella Ritorto}
\address{Mathematisch-Geographische Fakult\"at, Katholische Universit\"at Eichst\"att-Ingolstadt, 85071 Eichst\"att, Germany}
\email{antonella.ritorto@ku.de}

\author{Elvira Zappale}
\address{Dipartimento di Scienze di Base ed Applicate per l'Ingengeria/ Sapienza - Universit\`a di Roma, 00161, Rome, Italy}
\email{elvira.zappale@uniroma1.it}

\begin{document}

 
\maketitle

 \begin{abstract}
 	Motivated by the direct method in the calculus of variations in $L^{\infty}$, our main result identifies the notion of convexity characterizing the weakly$^*$ lower semicontinuity of nonlocal supremal functionals: Cartesian level convexity.
 	This new concept coincides with separate level convexity in the one-dimensional setting and is strictly weaker for higher dimensions. We discuss relaxation in the vectorial case, showing that the relaxed functional will not  generally maintain the supremal form. Apart from illustrating this fact with examples of multi-well type, we present precise criteria for structure-preservation. When the structure is preserved,  
 	a representation formula is given in terms of the Cartesian level convex envelope of the (diagonalized) original supremand. This work does not only complete the picture of the analysis initiated in [Kreisbeck \& Zappale, Calc.~Var.~PDE, 2020], but also establishes a connection with double integrals. 
 	We relate the two classes of functionals via an $L^p$-approximation in the sense of $\Gamma$-convergence for diverging integrability exponents.   
 	The proofs exploit recent results on nonlocal inclusions and their asymptotic behavior, and use tools from Young measure theory and convex analysis. 
    
\vspace{8pt}


 
 
  \noindent\textsc{Date:} \today.
 \end{abstract}


\section{Introduction}\label{sec:introduction}

The existence theory for minimizers in the calculus of variations is closely linked to convexity notions. As coercivity is usually achieved in function spaces endowed with weak($^\ast$) topologies, the central task to guarantee the applicability of the direct method lies in verifying the (sequential) weak($^\ast$) lower semicontinuity of the functionals. Over the decades, substantial effort has been put into finding necessary and sufficient conditions for the latter. Depending on the class of functionals at hand, different types of generalized convexity come into play. 
Fundamental results that identify the correct notions (under suitable growth assumptions) include in the classical integral setting: 
convexity for integral functionals on Lebesgue spaces~\cite{Att84, ButDal83},
quasiconvexity for integral functionals on Sobolev spaces with dependence on the weak gradients~\cite{Acerbi-Fusco, Mor66}, 
 $\Acal$-quasiconvexity for integral functionals subject to first-order PDE constraints~\cite{FoM99}, and in the branch of variational $L^\infty$-problems~\cite{Ar65, KatBook}:  
level convexity for supremal functionals on $W^{1,\infty}$ with supremands depending on weak gradients in the scalar case~\cite{BaLiu97, Prinari09}, and strong Morrey quasiconvexity in the corresponding vectorial case~\cite{Barron-Jensen-Wang}.

If functionals fail to be weak($^\ast$) lower semicontinuous, the existence of minimizers is not readily  guaranteed. In this case, the study of the relaxed problem, for which one replaces the functional by its weak($^\ast$) lower semicontinuous envelope, allows deducing helpful information about the limits of minimizing sequences. The difficulty is then to find explicit representation formulas for the relaxed functionals that are sufficiently easy to work with. This includes in particular to answer the question whether the formulas are structure-preserving, meaning whether they are again of the same type as the original functionals. The overall rationale for integral and supremal functionals in case structure-preservation holds is then that the relaxation can be obtained, from suitable convexification of its integrand and supremand, respectively; for an introduction to relaxation theory, see e.g.~\cite{Dac08, Rindler}.

Over the last years, the study of variational models with nonlocal features has attracted increased interest in the community, motivated by the desire to develop a solid understanding of 
global effects, long-range interactions, and singular behavior in physical phenomena and technical applications, which standard local modeling approaches cannot capture. To mention but a few selected examples from the recent literature, functionals with a nonlocal character appear in the theory of phase transitions~\cite{DMFoLe18, SaVa12},
in peridynamics~\cite{BMC, MenDu15}, in new models of hyperelasticity~\cite{BeCuMC2020, BeCuMC20},  in image processing~\cite{BrNg18, dEdLRTr20}
or in machine learning applications~\cite{AnDiKh20, HoKu20}. 
From the mathematical perspective, the presence of nonlocality in variational problems requires substantially different techniques from standard ones, which often rest on localization arguments and are therefore not applicable.  Regarding double-integral functionals, essential insights have been established by now.
Yet, there remain gaps in the literature, especially when it comes to relaxation theory; we give a brief overview of the current state of the art towards the end of the introduction, see also Table~\ref{tab1}.  
Another type of nonlocal integral functionals, namely, those involving Riesz fractional gradients are studied in~\cite{KrS22} (cf.~also~\cite{BeCuMC20, Shieh-Spector}). In this setting, a suitable translation procedure between classical and fractional gradients shows in particular that quasiconvexity of the integrand provides the correct condition for weak lower semicontinuity.   

This paper revolves around a class of nonlocal variational problems in $L^\infty$. Precisely,  our objects of interest are nonlocal homogeneous supremal functionals of the form
\begin{equation}\label{es}
	J_W(u)=\esssup_{(x, y)\in \Omega\times \Omega} W(u(x),u(y)), \qquad u\in L^{\infty}(\Omega;\R^m),
\end{equation}
where $\Omega\subset \R^n$ is a bounded open set, $m\in \N$ and $W\colon \R^m\times\R^m \to \R$ is (mostly assumed to be) lower semicontinuous and coercive; observe that the functional $J_W$ is invariant under symmetrization and diagonalization of its supremand $W$, as first shown in~\cite[Section 7.1]{KrZ19}, i.e., 
\begin{center}
	$J_W=J_{\widehat W}$ \qquad with $\widehat W(\xi, \zeta) =\max W(\{\xi, \zeta\}\times\{\xi, \zeta\})$ for $(\xi, \zeta)\in \R^m\times \R^m$; 
\end{center} 
for more details on the role of diagonalization see Section~\ref{subsec:diagmax}.
As our main contributions, we characterize the $L^{\infty}$-weak$^\ast$ lower semicontinuity of $J_W$  in terms of a new convexity condition on the symmetrized and diagonalized supremand $\widehat W$ and determine criteria for structure preservation under relaxation, see~Theorems~\ref{theo:charact-wls} and~\ref{theo:relaxation} below for the precise statements. It follows in particular that relaxations may generally not be supremal functionals anymore in the vectorial setting $m>1$. 
Our analysis provides a comprehensive result, unifying and extending the work initiated in~\cite{KrZ19}. There, two of the authors solve the scalar case ($m=1$) and consider the vectorial case under rather restrictive assumptions, which, in particular, force the relaxations to be again of the type~\eqref{es}. The 
arguments in~\cite{KrZ19} build up around the notion of separate level convexity in the vectorial components.  Inspired by the local setting of $L^\infty$-functionals, where one studies functionals \begin{equation*}
	L^{\infty}(\Omega;\R^m) \ni u\mapsto \esssup_{x \in \Omega} f(u(x) ) 
\end{equation*}
with a coercive and lower semicontinuous $f:\R^m\to \R$, this may appear like a natural choice. Indeed, as already mentioned above, in this local setting, level convexity of the supremand $f$ is known to be a necessary and sufficient condition for weak$^*$ lower semicontinuity~\cite{BaLiu97}, and relaxation is again a supremal functional with the level convexification of $f$ as the supremand~\cite{Prinari06}, cf.~also~\cite{Acerbi-Buttazzo-Prinari, CarPiPri05}. \color{black}

We show, however, that separate level convexity does not constitute the appropriate concept for characterizing the weak$^\ast$ lower semicontinuity of nonlocal supremals in dimensions greater than one.  Addressing the general vectorial case requires a change of perspective along with a new concept of convexity for sets, which we call \textit{Cartesian convexity}:
\begin{center}
	A set $E\subset \R^m\times \R^m$ is Cartesian convex if $A\times A\subset E$ implies $A^{\rm co}\times A^{\rm co}\subset E$; 
\end{center}
see~Definition~\ref{def:cartesianseparateconvexity} and Lemma~\ref{lem:alternative}. Consequently, a function $W:\R^m\times \R^m\to \R$  is said to be Cartesian level convex, if all its sublevel sets are Cartesian convex; an equivalent representation via a Jensen-type inequality is presented in~Proposition~\ref{prop:equiv-wslc}.
Note that Cartesian (level) convexity is indeed identical with separate (level) convexity if $m=1$ and is strictly weaker for higher dimensions, cf.~Remark~\ref{rem:comparison}.  

With this terminology at hand, our key characterization result now reads as follows. 

\begin{theorem}[Characterization of weak$^\ast$ lower semicontinuity]\label{theo:charact-wls} Let $W:\R^m\times \R^m\to \R$ be lower semicontinuous and coercive, and let $J_{W}$ be defined as in~\eqref{es}. Then, $J_W$ is (sequentially) weakly$^\ast$ lower semicontinuous in $L^\infty(\Omega;\R^m)$ if and only if 
	\begin{center}
		$\widehat W$ is Cartesian level convex,
	\end{center}
	i.e., all sublevel sets of $\widehat W$ are Cartesian convex.
\end{theorem}

Notice that the previous result can equivalently be stated for weak lower semicontinuity in $L^\infty$ and its sequential version, considering that the $L^\infty$-weak$^\ast$ topology admits a metrizable description on bounded sets, cf.~e.g.~\cite[A.1.5]{Fonseca-Leoni-book}; without further mentioning, we use the same reasoning throughout the paper.  

The proof of~Theorem~\ref{theo:charact-wls} passes through the theory of nonlocal inclusions. Indeed, by an adaption of the analogous statement for local supremals in~\cite{Acerbi-Buttazzo-Prinari} (see~\cite[Proposition 7.1]{KrZ19}), \color{black} the  weak$^\ast$ lower semicontinuity of  $J_W$ is equivalent to the weak$^\ast$ closedness of all its sublevel sets. 
The latter are given exactly by the functions $u\in L^\infty(\Omega;\R^m)$ that solve an inclusion problem
\begin{align}\label{nonlocalinclusion_intro}
	(u(x), u(y)) \in K \quad \text{for a.e. } (x,y)\in \Omega\times\Omega 
\end{align}
with a compact set $K\subset \R^m\times \R^m$, where $K$ is chosen as the sublevel set $L_c(W)$ of the supremand $W$ for each level $c\in \R$.
With this viewpoint, Theorem~\ref{theo:charact-wls} is a consequence of the following statement of independent interest: the set of solutions to~\eqref{nonlocalinclusion_intro} remains invariant under weak$^*$-limits if and only if the symmetrization and diagonalization of $K$, i.e., 
\begin{center}
	$\widehat{K}=:\{(\xi,\zeta)\in K: \{\xi, \zeta\}\times \{\xi, \zeta\}\in K\}$,
\end{center} is Cartesian convex, see Proposition~\ref{prop:weakstar_closure}. In showing this, we benefit from recent insights into alternative representations of the solution sets and their asymptotic behavior 
as established in~\cite{KrZ19, Kreisbeck-Zappale-2019Loss}. This allows us to  
break the complexity down to the study of nonlocal inclusions that are much easier to handle, namely, problems of the form~\eqref{nonlocalinclusion_intro} with Cartesian squares of the form $A\times A\subset \R^m\times \R^m$ as constraining sets.  

\smallskip

Whenever $J_{W}$ fails to be weakly$^*$ lower semicontinuous in $L^\infty(\Omega;\R^m)$, we pass
to the relaxed problem, for which one replaces the functional by its weak$^\ast$ lower semicontinuous envelope; precisely, the relaxation of $J_W$ is given by
\begin{align}\label{JW_rlx}
	J_W^{\rm rlx}(u) := \inf\bigl\{\liminf_{j\to \infty}{J_W(u_j)}: u_j\weaklystar u \ \text{in $L^\infty(\Omega;\R^m)$}\bigr\} 
\end{align} 
for $u\in L^\infty(\Omega;\R^m)$. For the scalar case $m=1$, it was recently established in~\cite[Theorem~1.3]{KrZ19} that 
\begin{align}\label{1d22}
	J_W^{\rm rlx}=J_{\widehat{W}^{\rm slc}},
\end{align} 
where the relaxed supremand $\widehat W^{\rm slc}$ is the separate level convexification of $\widehat W$. This shows, in particular, that the supremal structure of the functional is invariant under relaxation.
The argument in~\cite{KrZ19} relies on the observation that the separately convex hull of  a compact set $K$ (in particular,~any sublevel set $L_c(W)$ with $c\in \R$) has a specifically simple structure, namely it can be represented as a union of basic squares with corners in $K$, in formulas,
\begin{align*}
	K^{\rm sc} = \bigcup_{(\xi, \zeta)\in K} [\xi, \zeta]\times [\xi, \zeta],
\end{align*} 
and that all maximal Cartesian products of $K^{\rm sc}$ are necessarily contained in $\{[\xi, \zeta]\times [\xi, \zeta]: (\xi, \zeta)\in K\}$. 
An analogous property is no longer true in the vectorial case with $m>1$, cf.~\cite[Remark~4.8\,c)]{KrZ19}.  This observation, when applied to the sublevel sets of $W$, can be viewed as the origin for the structural change that $J_W$ may undergo during relaxation. 
For illustration, we discuss two explicit examples of $L^\infty$-functionals of multi-well type, one with and one without structure preservation, see~Example~\ref{ex:preservation} and~\ref{ex:counterexample} respectively. 
The latter relies on what we refer to as the effect of hidden Cartesian squares. It is related to the fact that a union of Cartesian squares 
$$
 \bigcup_{i\in \Ical}A_i\times A_i 
$$
with an index set $\Ical$ and $A_i\times A_i \subset \R^m\times \R^m$ for $i\in \Ical$ 
can contain $B\times B$ with $B\neq A_i$ for all $i\in \Ical$, for details see~Definition~\ref{def:hidden}.  
A condition ruling out the existence of such hidden Cartesian squares is the concept of a \textit{basic Cartesian convexification}, as introduced in~Definition~\ref{def:condition(S)}. Essentially, it says that every maximal Cartesian square of the Cartesian convex hull of a set in $\R^m\times \R^m$ coincides with the (classical) convexification of a maximal Cartesian square of the original set. 
In fact, it turns out that  this property for the sublevel sets of $W$ gives a full characterization of the cases when relaxation of $J_W$ is representable as~in~\eqref{es}. 
\color{black}
\begin{theorem}[Relaxation of nonlocal supremal functionals]\label{theo:relaxation}  Let $J_W$ and $J_W^{\rm rlx}$ be as in~\eqref{es} and~\eqref{JW_rlx} with $W:\R^m\times \R^m\to \R$ lower semicontinuous and coercive.  If every sublevel set of $\widehat{W}$ has a basic Cartesian convexification according to Definition~\ref{def:condition(S)}, 
then
\begin{align*}
	J_W^{\rm rlx}= J_{\widehat{W}^{\times\rm lc}},
\end{align*}
where $\widehat{W}^{\times\rm  lc}$ stands for the Cartesian level convex envelope of $\widehat{W}$. Otherwise, there exists no lower semicontinuous and coercive $G\colon \R^m\times\R^m\to\R$ such that $J_W^{\rm rlx} = J_G$ with $J_G$ as in~\eqref{es}. 
\end{theorem}

We remark finally that every subset of $\R\times \R$ has a basic Cartesian convexification, so that, in correspondence with~\eqref{1d22}, $J^{\rm rlx}_W=J_{\widehat{W}^{\times \rm lc}}=J_{\widehat W^{\rm slc}}$ holds in the one-dimensional setting without further restriction. 

\smallskip

To place the results above into a broader context, we establish a relation between nonlocal supremal functionals of the form~\eqref{es}, which are defined on essentially bounded functions, with homogeneous double integrals on $L^p$-spaces, that is, 
\begin{align}\label{double-integral_intro}
	L^p(\Omega;\R^m) \ni u\mapsto \int_\Omega\int_\Omega V(u(x), u(y)) \,dx\, dy
\end{align}
with $V:\R^m\times \R^m\to \R$ lower semicontinuous with standard $p$-growth. 
The existence theory for minimizers of this class of functionals was developed over the last two decades. In  a series of papers, starting with~\cite{Ped97} and continued in \cite{BMC, BePe06, Munoz, Pedregal2016, Pe21}, separate convexity of the double integrand has been identified as the necessary and sufficient condition for their weak lower semicontinuity; for a discussion of conditions in the inhomogeneous case, we refer to~\cite{BMC}. When it comes to the relaxation of~\eqref{double-integral_intro}, a comprehensive understanding of the problem is still to be developed.
Recent counterexamples in~\cite{KrZ19, TeMC20} indicate that even in the one-dimensional case, weak lower semicontinuous envelopes of functionals like~\eqref{double-integral_intro} may belong to a different class in general.  Beyond these specific examples, though, neither explicit formulas nor a general characterization of structure preservation under relaxation are currently available. To what extent, the results of this work can shed light on these open questions remains to be explored in future research.

It is well-known that the standard $L^p$-norm converges to the $L^{\infty}$-norm as $p$ goes to infinity. More generally, $L^p$-approximations hold in the local setting, connecting classical integral and supremal functionals approximatively in the limit of diverging power-law exponents, see e.g.~\cite{Ansini-Prinari, Champion-DePascale-Prinari, Eleuteri-Prinari, Pri15, PriZa20}.  
They have been used as technical tools, e.g.~for proving the existence of absolute minimizers of supremal functionals~\cite{BaJeWa01, Champion-DePascale-Prinari}, a homogenization result of $L^\infty$-functionals~\cite{Briani-Garroni-Prinari} and for the mathematical analysis of dielectric breakdown in composite materials~\cite{Garroni-Nesi-Po}. 

Stimulated by the local power-law approximations, we provide here a parallel statement in the nonlocal setting. Precisely, we perform a limit process of diverging integrability constants in terms of $\Gamma$-convergence~(cf.~Definition~\ref{def:Gamma}), which guarantees the convergence of almost-minimizers and infima of double integral functionals to minimizers and minima of a $\Gamma$-limit of nonlocal supremal type.

\begin{theorem}[Nonlocal \boldmath{$L^p$}-approximation]\label{theo:gral-vectorial-case} Let $q>1$ and $W\colon \R^m\times\R^m \to [0,\infty)$ be 
	lower semicontinuous with standard $q$-growth, i.e., there exist constants $c_1, c_2, C>0$ such that
	\begin{align}\label{lineargrowthandcoercivity}
		c_1 |(\xi,\zeta)|^q-c_2\leq W(\xi,\zeta)\leq C \,\bigl(|(\xi,\zeta)|^q+1\bigr) \qquad \text{for $(\xi,\zeta)\in \R^m\times \R^m$.}
	\end{align}
	For $p\ge q$, let $I_W^p:L^q(\Omega;\R^m)\to [0, \infty]$ be defined as 
	\begin{align*}
		I_W^p(u) = \begin{cases}
			\left(\displaystyle \int_\Omega\int_\Omega W^p(u(x), u(y)) \,dx\, dy\right)^{\frac{1}{p}}  & \text{ if } u\in L^p(\Omega;\R^m), \\
			\infty & \text{ otherwise}.
		\end{cases}
	\end{align*}
	Then, $(I_W^p)_p$ $\Gamma$-converges with respect to the weak topology in $L^q(\Omega;\R^m)$ as $p\to\infty$ to the functional $I_W^\infty:L^q(\Omega;\R^m)\to [0, \infty]$ given by 
	\begin{align*}
		I_W^\infty(u) = \begin{cases}  
			J_{W}^{\rm rlx}(u)
			& \text{ if } u\in L^\infty(\Omega;\R^m), \\
			\infty & \text{ otherwise}. 
		\end{cases}
	\end{align*}
	
Moreover, any sequence $(u_p)_p\subset L^q(\Omega;\R^m)$ with $\sup_{p\geq q} I_W^p(u_p)<\infty$ is relatively sequentially compact in the weak topology of $L^q(\Omega;\R^m)$ and the limits of weakly converging subsequences lie in $L^\infty(\Omega;\R^m)$. 
\end{theorem} 

Notice that the previous result does not require any type of convexity assumptions on~the double-integrand $W$, and is new even in the one-dimensional setting. 
The proof strategy for the compactness and the construction of recovery sequences (cf.~Definition~\ref{def:Gamma}) is rather standard. 
To obtain the lower bound of the $\Gamma$-limit, we take inspiration from~\cite{Champion-DePascale-Prinari}, while exploiting recent results about the Young measure representation of the weak$^\ast$ closures of solution sets to nonlocal inclusions (see~Remark~\ref{rem:YM_representation}, and also~\cite{Kreisbeck-Zappale-2019Loss}).

Even though, the two types of functionals - double-integrals and nonlocal supremals - are linked asymptotically in the limit of diverging integrability constants, they show qualitatively very different features, as the table in Figure~\ref{tab1} summarizes at a glance. 
\color{black}
\begin{table}[h!]
	\begin{center}
		\color{black}
		\label{tab:table1}
		\begin{tabular}{|l|r|r|l}
			\hline\hspace{0.2cm}
			& \textbf{Double integrals} & \textbf{Nonlocal supremals}\phantom{abbbbb}\\ 
			\hline\hline
			invariance under symmetrization  & Yes & Yes \\ \hline
			invariance under diagonalization  & No & Yes \\ \hline
			necessary and sufficient condition  & separate convexity &  Cartesian level convexity for $m\geq 1$ \footnotesize [$\ast$] \\  for weak$^{(\ast)}$ lower semicontinuity& \footnotesize \cite{BMC, BePe06, Munoz, Pedregal2016} \color{blue}\normalsize & separate level convexity for $m=1$\footnotesize{ \cite{KrZ19} }\\ \hline
			structure preservation  & No \footnotesize \cite{Kreisbeck-Zappale-2019Loss, TeMC20} & Yes for $m=1$ \footnotesize{\cite{KrZ19}} \\ under relaxation &  & No for $m>1$\footnotesize \phantom{4}  [$\ast$] \\ \hline
			criteria for   & currently still open & basic Cartesian level convexification \footnotesize   [$\ast$] \\ structure preservation & &  \\ \hline
		\end{tabular}
		\vspace{0.2cm}
	\end{center}
	\caption{Comparing properties of double-integrals and nonlocal supremals. Results marked with [$\ast$] are proven in this paper.}\label{tab1}
\end{table}

\bigskip

\color{black} This paper is organized as follows. After the introduction, we start in Section~\ref{sec:prelim} by fixing notations and gathering some technical tools from asymptotic analysis as well as useful observations about maximal Cartesian squares; other auxiliary results that require more specific terminology have been moved to the appendix.  
Section~\ref{sec:wsc-sets} is devoted to the definition and analysis of the new notion of Cartesian (level) convexity, for both sets and functions, as well as the corresponding Cartesian convex hulls and envelopes. In particular, we prove different alternative characterizations, including a Jensen-type formula in the case of functions, discuss some relevant properties, and provide a comparison with standard notions of convexity, precisely with separate convexity.
In Section~\ref{sec:inclusions}, Cartesian convexity is identified as the necessary and sufficient condition   for the weak$^\ast$ closedness of the solution set for nonlocal inclusions, and we investigate when the weak$^\ast$ closures give rise to a nonlocal inclusion of the same type, cf.~Propositions~\ref{prop:weakstar_closure} and~\ref{prop:structure_inclusions}, respectively. 
These findings are then transferred to the context of nonlocal supremal functionals in Section~\ref{sec:lower-rlx}, where the proofs of our main results Theorem~\ref{theo:charact-wls} and Theorem~\ref{theo:relaxation} are presented. 
We illustrate the two scenarios of structure preservation during relaxation and the loss there of with examples of multi-well supremals. 
Finally, Section~\ref{sec:Lp-approx} contains the proof of Theorem~\ref{theo:gral-vectorial-case} and builds a rigorous bridge between the variational theories for nonlocal supremals and double integrals.


\section{Preliminaries}\label{sec:prelim}

We use this section to introduce notation, to recall some well-known concepts, and to prove a few auxiliary results that will be useful in the rest of the article. 

\subsection{Notation}\label{subsec_notation} 
Let $m\in \N$, $\xi, \zeta\in \R^m$, and $| \cdot |$ be the Euclidean norm in $\R^m$. We denote by $[\xi,\zeta]=\{t\xi +(1-t)\zeta \colon t\in [0,1]\}$ the segment in $\R^m$ with extreme points $\xi, \zeta$. The set $B_c(a)\subset \R^m$ is the Euclidean ball with center point $a$ of radius $c$. For a set $A \subset \R^m$, $\mathbbm{1}_A$ is the characteristic function of $A$, meaning $\mathbbm{1}_A =1$ in $A$ and $\mathbbm{1}_A =0$ in $\R^m \setminus A$. By $A^{\rm co}$, 
we mean the convex hull of $A$. The distance between a point $\xi \in\R^m$ and a set $A\subset \R^m$ is $\dist(\xi,A)=\inf_{\zeta \in A}|\xi-\zeta|$. In $\R^m\times \R^m$, we use the norm $|(\xi,\zeta)|=\max\{|\xi|, |\zeta|\}$. Then, for $A, B\subset \R^m$, 
\begin{align}\label{distmax}
	\dist \bigl((\xi, \zeta), A\times B\bigr) = \max\{\dist(\xi, A), \dist(\zeta, B)\},
\end{align}
and $d_H(A, B) = \max\{\sup_{\xi \in B} \dist(\xi, A), \sup_{\zeta\in A} \dist(\zeta, B)\}$ is the Hausdorff distance between $A$ and $B$.

Let $\mathcal M(\mathbb R^m)$ be the space of positive finite Radon measures and by $\mathcal Pr(\mathbb R^m)$, the subset of probability measures. The support of a measure $\nu\in \Mcal(\R^m)$, in formulas $\supp \nu$, consists of all the points $\xi\in \R^m$ such that $\nu(O)>0$ for every open neighborhood $O\subset \R^m$ of $\xi$.
For $\nu \in \mathcal{M}(\R^m)$, we write 
$$
[\nu] = \langle \nu, {\rm id} \rangle = \int_{\R^m} \xi \, d\nu(\xi)
$$
for its barycenter, and recall that the convex hull of any compact set $C\subset \R^m$ can be represented as \begin{equation}\label{convexhull}
	C^{\rm co} = \{[\nu]: \nu\in \mathcal Pr(\R^m),\ \supp \nu \subset C\},
\end{equation}
cf.~e.g.~\cite{Dol03}. The product measure of $\nu, \mu \in \mathcal{M}(\R^m)$ is denoted by $\nu \otimes \mu$. 

Let $1\le p\le \infty$, $\nu\in \Mcal (\R^m)$, and $U \subset \R^n$ be a bounded open set, then  $L^p_{\nu}(U;\R^m)$ stands for the standard  Lebesgue spaces. If $1\le p<\infty$ and $(u_j)_j \in L^p_{\nu}(U;\R^m)$ is a weakly convergent sequence with limit $u\in L^{p}_{\nu}(U;\R^m)$, we write $u_j \rightharpoonup u$ in $L^p_{\nu}(U;\R^m)$. Analogously, $u_j \rightharpoonup^\ast u$ in $L^{\infty}_{\nu}(U;\R^m)$ stands for the weak$^*$ convergence in $L^\infty_{\nu}(U;\R^m)$. When $m=1$, we simply write $L^p_{\nu}(U)$ and $L^{\infty}_{\nu}(U)$, if $\nu$ is the Lebesgue measure, $L^p(U;\R^m)$ and $L^{\infty}(U;\R^m)$. 

For $A\subset \R^m$, and $\nu\in \Mcal(\R^m)$, the $\nu$-essential supremum in $A$ of a Borel function $f:\R^m\to \R$ is 
\begin{equation*}
	\nu\text{-}\esssup_{\xi\in A} f(\xi) = \inf\{c\in \R: f\leq c \ \text{$\nu$-a.e. in $A$}\} =
	\inf_{N\subset A, \nu(N)=0}\sup_{\xi\in A\setminus N} f(\xi);
\end{equation*}
in short-hand, $\nu\text{-}\esssup_A f$, or simply $\nu\text{-}\esssup f$ in case $A=\R^m$. If $f$ is lower semicontinuous, then
\begin{equation}\label{ess-supp}
	\nu\text{-}\esssup_{\xi\in \R^m} f(\xi) = \sup_{\xi\in \supp \nu} f(\xi).
\end{equation}

Let $X$ be a normed space and $V:X\to \R$. Then the sublevel set of $V$ at level $c\in \R$ is denoted by
\begin{align*}
	L_c(V)=\{x\in X \colon V(x)\le c \}.
\end{align*}
As a simple consequence of this definition, $V$ can be expressed as
\begin{equation}\label{levelset-i}
	V(x) = \inf\{ c\in \R \colon x \in L_c(V)\} \quad \text{ for } x\in X.	
\end{equation}
On the other hand, if a family of monotone increasing sets $(F_c)_{c\in\R} \subset X$  satisfies
$\bigcap_{j\in\N} F_{c_j} =F_c $ for any decreasing real sequence $(c_j)_j$ with $c_j\to c$ as $j\to \infty$  and 
$V(x) = \inf\{ c\in \R \colon  x\in F_c \}$ for $x\in X$, then 
\begin{equation}\label{level-sets-ii}
	F_c=L_c(V) \quad \text{ for every $c\in \R$.}
\end{equation}

\subsection{Tools from asymptotic analysis}\label{subsectGamma} We start by recalling some elements from Young measure theory; for more on the topic, see e.g.~\cite{Fonseca-Leoni-book, Ped97-book}. 
A map $\nu \colon U \subset \R^n \to \Mcal(\R^m)$ is called a Young measure if it is essentially bounded and weakly$^\ast$ measurable with $\nu_{x} = \nu(x) \in \mathcal{P}r(\R^m)$ a.e. $x \in U$; we write $\nu \in L^{\infty}_{w}(U; \mathcal{P}r(\R^m))$.

If $(v_j)_j\subset L^p(U;\R^m)$ with $p>1$ and $\nu\in L^\infty_w(U;\mathcal{P}r(\R^m))$, the sequence $(v_j)_j$ generates the Young measure $\nu$, 
in formulas, 
\begin{center}
	$v_j\stackrel{\rm YM}{\longrightarrow} \nu$ as $j\to \infty$,
\end{center}
if 
$$
\lim_{j\to \infty} \int_U f(x) \varphi(v_{j}(x)) \, d x = \int_U f(x)  \int_{\R^m} \varphi(\xi) \, d\nu_x(\xi) \, dx
$$
for every $f\in L^1(U)$ and every $\varphi\in C_0(\R^m)$ in the closure of real-valued functions on $\R^m$ with compact support. 

The next proposition is a version of the fundamental theorem of Young measures (see \cite{Fonseca-Leoni-book, Ped97-book}), as we use it in the proof of Theorem~\ref{theo:gral-vectorial-case} in Section~\ref{sec:Lp-approx}. 
\begin{proposition}\label{prop:YM}
	Let $q>1$ and $(v_j)_j \subset L^q(U;\R^m)$. 	
	\begin{itemize} 
		\item[$(i)$] If $(v_j)_j$ is bounded in $L^q(U;\R^m)$, then there exists a subsequence (not relabeled) and a Young measure $\nu\in L^\infty_w(U;\mathcal{P}r(\R^m))$ such that $v_{j} \stackrel{\text{\rm YM\ }}{\longrightarrow} \nu$ as $j\to \infty$.
		
		\item[$(ii)$] If $(v_j)_j \subset L^q(U;\R^m)$ generates a Young measure $\nu$, then 
		$$
		\liminf_{j\to \infty} \int_{U} G(x, v_j(x))\, dx \ge \int_U \int_{\R^m} G(x,\xi) d\nu_x(\xi)\, d x 
		$$
		for any nonnegative, normal integrand $G \colon U\times \R^m \to \R$ such that $(G(\cdot, v_j(\cdot)))_j$ is equi-integrable. 
	\end{itemize}
\end{proposition}

The natural notion for capturing the asymptotic behavior of parameter-dependent variational problems is $\Gamma$-convergence of the associated functionals, which we state here in a setting relevant for this paper; a comprehensive introduction to the theory can be found in~\cite{DalMaso-book}. 

\begin{definition}[\boldmath{$\Gamma$}-convergence]\label{def:Gamma}
	Let $X$ be a Banach space with separable dual and $I_j: X\to [-\infty, \infty]$ for $j\in \N$ equi-coercive functionals, that is, $I_j\geq \Psi$ for every $j\in \N$ with $\Psi(x) \to  \infty$ as $\norm{x}_X\to \infty$, where $\norm{\cdot}_X$ stands for the norm in $X$. 
 	
Then $(I_j)_j$ $\Gamma$-converges to $I:X\to [-\infty, \infty]$ with respect to the weak toplogy in $X$, in formulas $I_j\stackrel{\Gamma}{\longrightarrow} I$ as $j\to \infty$ or $\Gamma$-$\lim_{j\to \infty} I_j=I$, if these two conditions hold:
	\begin{itemize}
		\item[$(i)$] for all $u \in X$ and for every sequence $(u_j)_j\subset X$ with $u_j\weakly u$ in $X$, 
		$$ 
		I(u)\leq \liminf_{j\to \infty} I_j(u_j); 
		$$
		
		\item[$(ii)$] for all $u \in X$ with $I(u)<\infty$ there exists a sequence $(u_j)_j \subset X$ 
		such that  $u_j\weakly u$ in $X$ and
		$$ 
		I(u)\ge \limsup_{j\to \infty} I_j(u_j).  
		$$
	\end{itemize}
\end{definition}

The main feature of the $\Gamma$-convergence is that it implies the convergence of minimizers of $(I_j)_j$ to minimizers of $I$, precisely,  if $u_j$ is a minimizer for $I_j$ for each $j\in \N$ and $u$ is a cluster point of $(u_j)_j$, then $u$ is a minimizer of $I$ and 
$$
I(u)=\lim_{j\to\infty}I_j(u_j),
$$
cf.~\cite[Cororally 7.20]{DalMaso-book}. 
When a family  $(I_p)_{p}$ of functionals $I_p:X\to [-\infty, \infty]$ with a real parameter $p$ is considered, the $\Gamma$-convergence $I_p\to I$ as $p\to \infty$ means that $I_{p_j}\to I$ as $j\to \infty$ for every sequence $(p_j)_j\subset \R$ converging to $\infty$. \color{black}\smallskip

We close this section with a well-known classic, the standard $L^p$-approximation of the $L^{\infty}$-norm. Let $\nu\in \mathcal Pr (\R^m)$ and $f\in L^p_\nu(\R^m)$ for $p>1$. Then, the map $p\mapsto \norm{f}_{L_\nu^p(\R^m)}$ is monotone increasing and
\begin{equation}\label{Lpapprox-nu}
	\lim_{p\to \infty} \norm{f}_{L^p_\nu(\R^m)} = \norm{f}_{L^\infty_\nu(\R^m)};
\end{equation}
in particular, $f\in L^\infty_\nu(\R^m)$ if and only if $\sup_{p>1} \norm{f}_{L^p_\nu(\R^m)}<\infty$. 
For any bounded measurable set $U\subset \R^m$, one obtain as a special case that
\begin{align}\label{Lpapprox_classic}
	\lim_{p\to \infty} \norm{f}_{L^p(U)}  = \norm{f}_{L^\infty(U)}. 
\end{align}

\color{black}
\subsection{Diagonalization and maximal Cartesian squares}\label{subsec:diagmax} A set $E\subset \R^m\times \R^m$ is called symmetric if $(\xi, \zeta)\in E$ if and only if $(\zeta, \xi)\in E$, and following~\cite{KrZ19}, we say that $E$ is diagonal if $(\xi,\xi), (\zeta, \zeta) \in E$ for every $(\xi, \zeta)\in E$. Let $\widehat{E}$ denote the diagonalization and symmetrization of $E$, that is,
\begin{equation*}
\widehat{E}= \bigl\{ (\xi, \zeta) \in E \colon  \{\xi,\zeta\}\times \{\xi,\zeta\}\subset E \bigr\}.
\end{equation*}

These concepts can be carried over to functions by applying the above definitions to their sublevel sets. Precisely, we introduce the symmetrization and diagonalization of $W\colon \R^m\times \R^m \to \R$ as
\begin{equation}\label{eq:W-hat}
\widehat{W}(\xi, \zeta) = \max \left\{ W(\xi, \zeta), W(\zeta, \xi), W(\xi, \xi), W(\zeta,\zeta) \right\} \quad \text{for $(\xi, \zeta) \in \R^m \times \R^m$.}
\end{equation}
It holds then for any $c\in \R$ that
\begin{align*}
L_c(\widehat{W}) = \widehat{L_c(W)}.
\end{align*} 
\color{black}

\color{black} A closely related notion to diagonality, which is central for our analysis, is that of maximal Cartesian squares. 
As indicated in the introduction, the identification of certain maximal Cartesian squares allows us to obtain explicit representations for weak$^\ast$ closures on nonlocal inclusions and relaxation formulas for nonlocal supremals. 
For an arbitrary $E\subset \R^m\times \R^m$, a set $B\times B\subset E$ with $B\subset \R^m$ is said to be a maximal Cartesian square of $E$ if every $A\subset B$ with  $A\times A \subset E$ satisfies $A=B$. We refer to the set of all maximal Cartesian squares of $E$ as $\Pcal_E$. Notice that  the existence of a maximal Cartesian square in $\Pcal_E$ containing an arbitrary square $A\times A\subset E$ is guaranteed by Zorn's Lemma.

It was shown in~\cite[Lemma 4.3]{KrZ19} 
that any symmetric and diagonal set $E\subset \R^m \times \R^m$ is equal to the union of its (maximal) Cartesian squares, i.e., 
\begin{align}\label{symdiag}
E = \bigcup_{A\times A \subset E}A\times A = 
 \bigcup_{B\times B\in \Pcal_E} B\times B.
\end{align}

On the other hand, if a subset of $\R^m\times \R^m$ can be written as the union of Cartesian squares, say 
\begin{align}\label{unionAi}
E=\bigcup_{i\in \Ical}A_i\times A_i 
\end{align}
with a (possibly uncountable) index set $\Ical$ and  nonempty sets $A_i\subset \R^m$ for $i\in \Ical$, it is clearly symmetric and diagonal. 
 However, in contrast to what may seem intuitive at first glance, the maximal Cartesian squares of $E$ do not necessarily coincide with Cartesian squares 
in~\eqref{unionAi},
meaning that in general, 
\begin{equation*}
\Pcal_E \not \subset \bigcup_{i\in \Ical}\{A_i\times A_i\}. 
\end{equation*}
The next example illustrates this effect.

\begin{example}\label{ex:hidden}
 Consider three sets $A_1, A_2, A_3\subset \R^m$ with no common element, but pairwise nonempty intersections, that is, $\bigcap_{i=1}^3 A_i=\emptyset$ and $A_i\cap A_j\neq \emptyset$ for $i, j\in \{1,2,3\}$ with $i\neq j$; see Figure~\ref{fig1a} for an illustration in two dimensions.  
With $E=  \bigcup_{i=1}^3A_i\times A_i$ and the nonempty set 
	\begin{equation}\label{eqM}
	M:=(A_1\cap A_2)\cup (A_2\cap A_3)\cup (A_3\cap A_1),
	\end{equation}
it is straightforward to check that $M\times M\subset  E$ and $M\neq A_i$ for $i\in \{1,2,3\}$. Hence, 
$\Pcal_E \not \subset  \bigcup_{i=1}^3\{A_i \times A_i\}$.  \color{black}
\begin{figure}[h]
		\begin{tikzpicture}
		\color{olive}
		
		\begin{scope}[xshift=0cm]
		\draw[color=blue!50!white, thick](1,2) -- (-0.5,-1);
		\draw[color=red!50!white, thick](-1,2) -- (0.5, -1);
		\draw[color=green!50!white, thick](1.5,1) -- (-1.5,1);
		
		\filldraw [black] (0.5,1) circle (1.5pt);
		\filldraw [black] (-0.5, 1) circle (1.5pt);
		\filldraw [black] (0,0) circle (1.5pt);
		
		\draw [black] (1.3,2.3) node {\color{blue} $A_1$};
		\draw [black] (-1.2,2.3) node {\color{red} $A_2$};
		\draw [black] (-1.7, 1.3) node {\color{green} $A_3$};
		
		\draw [black] (-1.7, 0) node {\color{black} $M$};
		\end{scope}			
		\end{tikzpicture}
		\caption{Illustration of sets $A_1, A_2, A_3$, and  $M$  as in Example~\ref{ex:hidden} for $m=2$ \color{black}.}\label{fig1a}
	\end{figure}
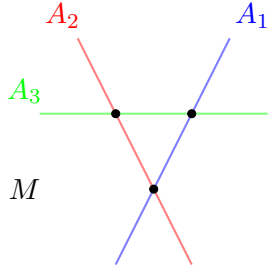
	\end{example}

\color{black}

Before following-up on this example, let us slide in three brief comments.
\begin{remark} 
	\label{rem:hidden} 
a) The primary reason for us to study sets of the form~\eqref{unionAi} is 
their relevance for our discussion of nonlocal supremal functionals of multi-well type in Section~\ref{subsec:multiwell}. 
In fact, the sublevel sets of the associated supremands have exactly this type of structure with a finite index set $\Ical$.

\color{black}

\smallskip		
b) We point out that the representation of $E\subset\R^m\times \R^m$ as in~\eqref{unionAi} in terms of a union of Cartesian squares is in general not unique, even if one supposes that each $A_i$ with $i\in\Ical$ is not contained in any $A_j$ with $j\in \Ical\setminus\{i\}$. As a simple example, take $m=1$ and
	\begin{align*}
	E=\bigcup_{(\alpha, \beta)\in [0,1]\times [0,1]}\{\alpha, \beta\}\times \{\alpha, \beta\} = [0,1]\times [0,1].
	\end{align*} 
\smallskip

c) Let $E\subset \R^m\times\R^m$ be as in~\eqref{unionAi}. 
If $A_i\setminus \bigcup_{j\in \Ical, j\neq i} A_j\neq \emptyset$ for some $i\in \Ical$, then $A_i\times A_i\in \mathcal P_E$. To show that the Cartesian square $A_i\times A_ i\subset E$ is actually maximal, suppose to the contrary that $C_i=A_i\cup \{\gamma\}$ with $\gamma\in A_k\setminus A_i$ for some $k\in \Ical$ satisfies $C_i\times C_i\subset E$. This implies that $(\xi, \gamma)$ with $\xi\in A_i\setminus \bigcup_{j\in \Ical, j\neq i} A_j$ lies neither in $A_i\times A_i$ nor in any other $A_j\times A_j$ with $j\in \Ical\setminus\{i\}$. Thus, $(\xi, \gamma)\in (C_i\times C_i)\setminus E=\emptyset$, which is a contradiction. 
\end{remark}

\medskip

With motivation from Example~\ref{ex:hidden}, we now introduce the following terminology.
\begin{definition}\label{def:hidden}
 Let $E$ be given as in~\eqref{unionAi}. Then, $M\times M\subset E$ is called a hidden Cartesian square of $E$ if $M\neq A_i$ for every $i\in \Ical$.
 If $M\times M\in \Pcal_E$, we speak of a hidden maximal Cartesian square.  
 \end{definition}

Even though the concept of hidden Cartesian squares is easily accessible in theory, finding ways to identify all of them explicitly for a given set is less obvious. The next result provides two basic conditions, a necessary one and a sufficient one, useful for detecting hidden Cartesian squares. 
\begin{lemma}\label{prop:hidden-squares}
Let $\Ical$ be an index set and $E=\bigcup_{i\in \Ical} A_i\times A_i$ with  nonempty $A_i\subset \R^m$ for $i\in \Ical$. 
Further, let $M_{\Ical'} = \bigcup_{i, j\in \Ical', i\neq j} A_i\cap A_j$ for any $\Ical'\subset \Ical$. 
\smallskip

$(i)$ If $A\times A \subset E$, then $A\subset A_i$ for some $i\in \Ical$ or
\begin{align*} 
A\subset M_\Ical = \bigcup_{i, j\in \Ical, i\neq j} A_i\cap A_j.
\end{align*} 
The latter means that each element of $A$ lies in at least two different sets $A_i$ and $A_j$ with $i, j\in \Ical$.\smallskip

$(ii)$ If $\Ical'\subset \Ical$ is such that one can find for any $l,i, j, k\in \Ical'$ with $i\neq j$ and $k\neq l$ an index $s=s(i, j, k, l)\in \Ical$ with
\begin{align}\label{assij}
(A_i\cap A_j)\cup (A_k\cap A_l)\subset A_{s},
\end{align}
then $M_{\Ical'}\times M_{\Ical'} \subset E$. 
\end{lemma}
\begin{proof} 	
 $(i)$ Arguing by contradiction, let $A\times A\subset E$ be nonempty and assume that for each $i\in \Ical$ there is $\zeta_i\in A\setminus A_i$, and that there exist $k\in \Ical$ and $\xi_k\in A$ with the property that $\xi_k\in A_k\setminus \bigcup_{j\in \Ical, j\neq k} A_j$. 
Then, $(\xi_k, \zeta_k)\notin A_k\times A_k$ as well as $(\xi_k, \zeta_k)\notin A_j\times A_j$ for all $j\in \Ical\setminus\{k\}$, and consequently, $(\xi_k,\zeta_k)\notin E$. This contradicts $(\xi_k, \zeta_k)\in A\times A\subset E$.\smallskip

\color{black}
$(ii)$ Let $(\xi, \zeta)\in M_{\Ical'}\times M_{\Ical'}$. Then, there are $i, j\in \Ical'$ and $k, l\in \Ical'$ with $i\neq j$ and $k\neq l$, such that $\xi\in A_i\cap A_j$ and $\zeta\in A_k\cap A_l$. By~\eqref{assij}, we find $s\in \Ical$ such that $(\xi, \zeta)\in A_s\times A_s\subset E$. 
\end{proof}

Finally, we collect some useful consequences of the previous lemma, 
 in particular, two simple cases where the existence of hidden Cartesian squares can be ruled out and a result for unions of three Cartesian squares. Note that for unions of four (or more) Cartesian squares, a variety of scenarios can occur, as illustrated in~Figure~\ref{fig4sets}. 
  
\color{black}
\begin{conclusion}\label{rem:hidden-squares}
Let $E=\bigcup_{i\in \Ical} A_i\times A_i$ with an index set $\Ical$ and nonempty sets $A_i\subset \R^m$ for $i\in \Ical$.  
Suppose that the sets $A_i$ with $i\in \Ical$ are as large as possible, in the sense that 
\begin{align}\label{ass12}
\bigcup_{i\in \Ical'}{A_i}\times \bigcup_{i\in \Ical'} A_i \not\subset E
\end{align}
for any index set $\Ical'\subset \Ical$ with cardinality greater than one. 
In particular,~\eqref{ass12} implies that $A_i\not\subset A_j$ for any $i, j\in \Ical$ with $i\neq j$. 
 The following statements are immediate consequences of~Lemma~\ref{prop:hidden-squares}. 	
\smallskip
	
	a) In case the sets $A_i$ with $i\in \Ical$ are pairwise disjoint, one has $M_\Ical=\emptyset$, and thus, $\Pcal_E = \bigcup_{i\in \Ical}\{ A_i\times A_i\}$ in view of Lemma~\ref{prop:hidden-squares}\,$(i)$ and Remark~\ref{rem:hidden}\,c).	\smallskip

	b) If $\Ical=\{1,2\}$, then $M_\Ical = A_1\cap A_2$ and we obtain $\mathcal P_{E}=\{A_1\times A_1, A_2\times A_2\}$, 
	i.e., $E$ does not have any hidden maximal Cartesian squares. \smallskip

	c) The assumption~\eqref{assij} is automatically fulfilled whenever $\# \Ical'\leq 3$. Therefore, if $\Ical=\{1,2,3\}$, the combination of Lemma~\ref{prop:hidden-squares}\,$(ii)$ and $(i)$ implies
	\begin{align*}
\Pcal_E\subset
\bigcup_{i=1}^3 \{A_i\times A_i\} \cup \{M_\Ical \times M_\Ical\}.
 \end{align*}
This shows that $E$ can have at most one hidden maximal Cartesian square. For an example with $\#\Pcal_E =4$, we refer back to Example~\ref{ex:hidden} and Figure~\ref{fig1a}.\smallskip
\color{red}
\end{conclusion}

 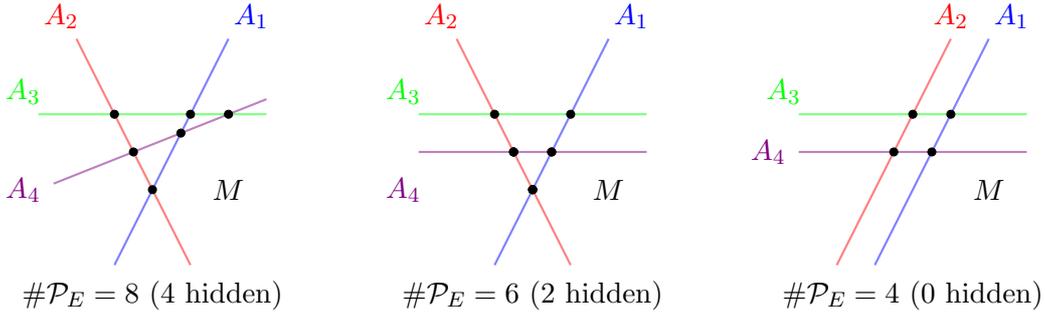
\begin{figure}[h]
 	\begin{tikzpicture}
 	\begin{scope}[xshift=-5cm]
 	\draw[color=blue!50!white, thick](1,2) -- (-0.5,-1); 
 	\draw[color=red!50!white, thick](-1,2) -- (0.5, -1); 
 	\draw[color=green!50!white, thick](1.5,1) -- (-1.5,1); 
 	\draw[color=violet!50!white, thick](-1.3,0.08) -- (1.5,1.2); 
 	
 	\filldraw [black] (0.5,1) circle (1.5pt);
 	\filldraw [black] (-0.5, 1) circle (1.5pt);
 	\filldraw [black] (0,0) circle (1.5pt);
 	\filldraw [black] (1,1) circle (1.5pt); 
	\filldraw [black] (0.374, 0.75) circle (1.5pt); 
	\filldraw [black] (-0.25, 0.5) circle (1.5pt); 

 	\draw [black] (1.3,2.3) node {\color{blue} $A_1$};
 	\draw [black] (-1.2,2.3) node {\color{red} $A_2$};
 	\draw [black] (-1.7, 1.3) node {\color{green} $A_3$};
 	\draw [black] (-1.7, 0) node {\color{violet} $A_4$};
 	
 	\draw [black] (1, 0) node {\color{black} $M$};
 	\draw [black] (0,-1.4) node {$\#\Pcal_E =8$ (4 hidden)};
 	\end{scope}

 	\begin{scope}[xshift=0cm]
 	\draw[color=blue!50!white, thick](1,2) -- (-0.5,-1); 
 	\draw[color=red!50!white, thick](-1,2) -- (0.5, -1); 
 	\draw[color=green!50!white, thick](1.5,1) -- (-1.5,1); 
 	\draw[color=violet!50!white, thick](-1.5,0.5) -- (1.5, 0.5); 

 	\filldraw [black] (0.5,1) circle (1.5pt);
 	\filldraw [black] (-0.5, 1) circle (1.5pt);
 	\filldraw [black] (0,0) circle (1.5pt);
 	\filldraw [black] (-0.25,0.5) circle (1.5pt);

 	\filldraw [black] (0.25,0.5) circle (1.5pt);
 	\filldraw [black] (-0.25, 0.5) circle (1.5pt);
 	\filldraw [black] (0,0) circle (1.5pt);
 	
 	\filldraw [black] (0,0) circle (1.5pt);
 	\filldraw [black] (-0.25,0.5) circle (1.5pt);
 	\filldraw [black] (0.25, 0.5) circle (1.5pt);
 		\draw [black] (1.3,2.3) node {\color{blue} $A_1$};
 	\draw [black] (-1.2,2.3) node {\color{red} $A_2$};
 	\draw [black] (-1.7, 1.3) node {\color{green} $A_3$};
 	\draw [black] (-1.7, 0) node {\color{violet} $A_4$};
 	
 	\draw [black] (1, 0) node {\color{black} $M$};
 	\draw [black] (0,-1.4) node {$\#\Pcal_E =6$ (2 hidden)};
 	\end{scope}

 	\begin{scope}[xshift=5cm]
 	\draw[color=blue!50!white, thick](-0.5,-1) -- (1,2); 
 	\draw[color=red!50!white, thick](-1,-1) -- (0.5, 2); 
 	\draw[color=green!50!white, thick](1.5,1) -- (-1.5,1); 
 	\draw[color=violet!50!white, thick](-1.5,0.5) -- (1.5, 0.5); 
 
	\filldraw [black] (0.5,1) circle (1.5pt);
	\filldraw [black] (0, 1) circle (1.5pt);
	\filldraw [black] (-0.25,0.5) circle (1.5pt);
 	\filldraw [black] (0.25, 0.5) circle (1.5pt);

 	\draw [black] (1.3,2.3) node {\color{blue} $A_1$};
 	\draw [black] (0.5,2.3) node {\color{red} $A_2$};
 	\draw [black] (-1.7, 1.3) node {\color{green} $A_3$};
 	\draw [black] (-1.9, 0.5) node {\color{violet} $A_4$};
 	 \draw [black] (1, 0) node {\color{black} $M$};
 	
 	\draw [black] (0,-1.4) node {$\#\Pcal_E =4$ (0 hidden)};
 	\end{scope}
 	\end{tikzpicture}
 	\caption{Number of maximal Cartesian squares for three examples of $E=\bigcup_{i=1}^4A_i\times A_i$ with $A_1, A_2, A_3, A_4\subset \R^2$ and the corresponding $M=M_{\{1,2,3,4\}}$.
 }\label{fig4sets}
 \end{figure}

\color{black}
\section{Cartesian convexity}\label{sec:wsc-sets}

This section revolves around our new notion of convexity for sets and functions, called Cartesian (level) convexity. As pointed out in the introduction, it plays a crucial role in characterizing the lower semicontinuity of nonlocal supremal and their relaxations. 
\color{black}
\subsection{Cartesian convexity of sets}\label{subsec:Cconvexsets}
Recall that a set $E\subset \R^m\times \R^m$ is separately convex (with vector components) if
\begin{align*}
[\xi_1, \xi_2]\times [\zeta_1, \zeta_2]\subset E \qquad \text{for all $(\xi_1,\zeta_1), (\xi_2, \zeta_2)\in E$ such that $\xi_1=\xi_2$ or $\zeta_1=\zeta_2$.} 
\end{align*}
The separately convex hull of $E$ is denoted by $E^{\rm sc}$.

With the motivation to introduce a suitable (in general) weaker concept of convexity for subsets of $\R^m\times \R^m$, we give the following definition. 
\begin{definition}[Cartesian convexity of sets]\label{def:cartesianseparateconvexity}
We say that $E\subset \R^m\times \R^m$ is Cartesian convex if 
\begin{align*}
[\xi_1, \xi_2]\times [\zeta_1, \zeta_2]\subset E 
\end{align*}
 for  all $(\xi_1,\zeta_1), (\xi_2, \zeta_2)\in E$ such that
$\{\xi_1, \xi_2, \zeta_1, \zeta_2\}\times \{\xi_1, \xi_2, \zeta_1, \zeta_2\}\subset E$ and $\xi_1=\xi_2$ or $\zeta_1=\zeta_2$.  
\end{definition}

The next lemma states useful alternative characterizations of Cartesian convexity.
\begin{lemma}\label{lem:alternative}
For any $E\subset \R^m\times \R^m$,
the following conditions are equivalent: 
\begin{itemize}
\item[$(i)$] $E$ is Cartesian convex; \\[-0.3cm]
\item[$(ii)$] $(A\times A)^{\rm co} =A^{\rm co}\times A^{\rm co}\subset E$ for all $A\subset \R^m$ with $A\times A\subset E$;\\[-0.3cm] 
\item[$(iii)$] any maximal Cartesian square of $E$ is convex.
\end{itemize}
If $E$ is symmetric and diagonal, then $(i)$-$(iii)$ are also equivalent to 
\begin{itemize}
\item[$(iv)$]
$E = \bigcup_{A\times A\subset E} A^{\rm co}\times A^{\rm co} 
= \bigcup_{B\times B\in \Pcal_E} B^{\rm co}\times B^{\rm co}$. 
\end{itemize}
\end{lemma}

\begin{proof} 
As for $(i)\Rightarrow (ii)$, let $A\times A\subset E$ and consider a maximal Cartesian square $B\times B\in \Pcal_E$ with 
 $A\subset B$. 
We claim that $B$ is convex. To show this, take $\xi_1, \xi_2\in B$ and observe that $(i)$ implies 
$([\xi_1,\xi_2]\times \{\alpha\})\cup (\{\beta\}\times [\xi_1,\xi_2])\subset E$ for any $\alpha, \beta\in B$, or in other words,
    	\begin{align}\label{774}
 ([\xi_1,\xi_2]\times B)\cup (B\times [\xi_1,\xi_2])\subset E.
    	\end{align}
   On the other hand, one can deduce from the Cartesian convexity of $E$ that
    	\begin{align*}
	[\xi_1,\xi_2]\times[\xi_1,\xi_2] \subset E;
    	\end{align*}
 indeed, for $\alpha\in [\xi_1,\xi_2]$ it follows from~\eqref{774} and $(i)$ that $([\xi_1,\xi_2]\times \{\alpha\})\cup (\{\alpha\}\times [\xi_1,\xi_2])\subset E$. 

\color{black} Summing up, we have that
    	$$
    	([\xi_1,\xi_2]\cup B)\times 	([\xi_1,\xi_2]\cup B) \subset E.
    	$$
	\color{black}
	Due to the maximality property of $B\times B$, this yields $[\xi_1, \xi_2]\subset B$, proving the convexity of $B$.  Consequently, $A^{\rm co}\subset B$ and therefore, $(A\times A)^{\rm co} = A^{\rm co}\times A^{\rm co}\subset B\times B\subset E$, as stated in~$(ii)$.\smallskip

The remaining implications are straightforward to verify. Indeed, $(ii)\Rightarrow (i)$ is a direct consequence of Definition~\ref{def:cartesianseparateconvexity}, the equivalence $(ii)\Leftrightarrow (iii)$ is clear considering the definition of maximal Cartesian squares, and $(ii)\Leftrightarrow (iv)$ holds in light of~\eqref{symdiag}.
\end{proof}

\color{black}
\begin{remark}[Comparison with separate convexity]\label{rem:comparison} Let $E\subset \R^m\times \R^m$ be symmetric and diagonal. It is evident that separate convexity is sufficient for Cartesian convexity. Whether it is also necessary depends on the dimension $m$. 
\smallskip

	a) In the scalar setting, where $m=1$, the set $E$ is Cartesian convex if and only if it is separately convex. Indeed, if $E\subset\R\times \R$ is Cartesian convex, we invoke~\cite[Lemma~4.7]{KrZ19} to find that the separately convex hull of $E$ can be expressed as the union of all symmetric squares with corners in the points of $E$, formally,
	\begin{equation*}
	E^{\rm sc} = \bigcup_{(\alpha, \beta)\in E} [\alpha, \beta]\times [\alpha, \beta]. 
	\end{equation*}
	Then, along with Lemma~\ref{lem:alternative}\,$(iv)$,
	\begin{align*}
	E^{\rm sc}= \bigcup_{\{\alpha, \beta\}\times \{\alpha, \beta\}\subset E} \{\alpha, \beta\}^{\rm co}\times \{\alpha, \beta\}^{\rm co}\subset \bigcup_{A\times A\subset E} A^{\rm co}\times A^{\rm co}=E,
	\end{align*} 
	which shows the separate convexity of $E$. 

	\smallskip
	
 b) If $m>1$, separate convexity is strictly stronger than Cartesian convexity, as this example illustrates. Consider the symmetric and diagonal set
	\begin{align*}
	E=(A_1\times A_1)\cup (A_2\times A_2),
	\end{align*}
	where $A_1, A_2\subset \R^m$ are two convex sets with nonempty intersection such that
	$(A_1\cup A_2)^{\rm co}\setminus (A_1\cup A_2)\neq \emptyset$. %
	Since $\mathcal P_E=\{A_1\times A_1, A_2\times A_2\}$ according to~Conclusion~\ref{rem:hidden-squares}\,b), the set $E$ satisfies the condition in~Lemma~\ref{lem:alternative}\,$(iii)$, and is thus Cartesian convex. To see that $E$ is not separately convex, we argue that
	$$
	E^{\rm sc} = E \cup \bigl[(A_1\cap A_2)  \times (A_1 \cup A_2)^{\rm co}\bigr] \cup \bigl[(A_1 \cup A_2)^{\rm co} \times (A_1 \cap A_2)\bigl] \,\neq E, 
	$$
	cf.~\cite[Remark 4.6\,b)]{KrZ19}. Observe that $E^{\rm sc}$ and $E$ do coincide after diagonalization, though, i.e.,~$\widehat{E^{\rm sc}}=E$.
\end{remark} 

\subsection{Cartesian convexification}\label{sec:Cartesianconvexification_set}
 If $E\subset \R^m\times \R^m$ is not Cartesian convex, 
 one may consider the smallest Cartesian convex set containing $E$, which we call the Cartesian convex hull of $E$ and denote by $E^{\times \rm c}$; \color{black} precisely, 
\begin{align*}
E^{\times \rm c} = \bigcap_{F\in \Fcal_E} F \qquad \text{with $\Fcal_E=\{F\subset \R^m\times \R^m: E\subset F, \text{$F$ is Cartesian convex}\}$.} 
\end{align*}
Note that \color{black} $E^{\times \rm c}$ is well-defined by the simple observation that intersections (even uncountable) of Cartesian convex sets are again Cartesian convex. 
If $E$ is symmetric and diagonal, so is $E^{\times \rm c}$.

In the spirit of \cite[Theorem 7.17]{Dac08} on the representation of separately convex hulls,
 we prove the following analogue 
for Cartesian convex hulls of symmetric and diagonal sets.
\begin{lemma}[Representation of \boldmath{$E^{\times \rm c}$}]\label{lem:well-definition-Ewsc}  Let $E\subset \R^m \times \R^m$ be symmetric and diagonal. Then,
		\begin{align*}
		E^{\times \rm c} = \bigcup_{k\in \N} E_k^{\times \rm c}
		\end{align*}
with inductively defined sets $E_0^{\times \rm c}:= E$ and
	\begin{align}\label{orderone}
	 E_k^{\times \rm c} := \bigcup_{A\times A\subset E_{k-1}^{\times \rm c}} A^{\rm co}\times A^{\rm co} = \bigcup_{B\times B\in \mathcal \Pcal_{E_{k-1}^{\times \rm c}}} B^{\rm co}\times B^{\rm co}\quad \text{ for $k\in \N$.}
	\end{align}
	 \end{lemma}	

\begin{proof} The argument is quite standard. Applying Lemma~\ref{lem:alternative}\,$(ii)$ iteratively 
shows that $E^{\times \rm c}$ contains all the nested sets $E_k^{\times \rm c}$ for $k\in \N$. 
The reverse inclusion follows from the Cartesian convexity of $\bigcup_{k\in \N} E_k^{\times \rm c}=:\widetilde E$. Indeed, given $(\xi_1,\zeta)$, $(\xi_2,\zeta)\in \widetilde{E}$ with $\{\xi_1,\xi_2,\zeta\}\times \{\xi_1,\xi_2,\zeta\}\subset \widetilde{E}$,
one can find $k\in \N$ such that
$\{\xi_1, \xi_2, \zeta\}\times \{\xi_1, \xi_2, \zeta\}\subset E_k^{\times \rm c}$, and hence, $\{\xi_1, \xi_2, \zeta\}^{\rm co}\times \{\xi_1, \xi_2, \zeta\}^{\rm co}\subset E_{k+1}^{\times \rm c}\subset \widetilde E$. In particular, $[\xi_1, \xi_2]\times\{\zeta\}\subset \widetilde E$, which concludes the proof. 
\end{proof}

Next, we determine explicitly the Cartesian convex hulls for two examples. The second one shows in particular that $E^{\times \rm c}_1$ is not Cartesian convex in general. 
\begin{example}\label{ex3}  
Let $E=\bigcup_{i=1}^3A_i\times A_i$ with $A_i\subset \R^m$ for $i=1, 2,3$. 
\smallskip

	a) Suppose that $A_i$ are convex and  $A_i\cap A_j\neq \emptyset$ for $i\neq j$. 
With $M= \bigcup_{i, j=1,2,3, i\neq j} A_i\cap A_j$, we obtain  
 $\Pcal_{E}\subset\bigcup_{i=1}^3\{A_i\times A_i\}\cup \{M\times M\}$ by Conclusion~\ref{rem:hidden-squares}\,c), and thus,
	\begin{align}\label{E1timesc}
	E_1^{\times \rm c} = \bigcup_{i=1}^3(A_i^{\rm co} \times A_i^{\rm co})\cup (M^{\rm co}\times M^{\rm co}).
	\end{align}
Lemma~\ref{prop:hidden-squares}\,$(i)$ then 
 shows that $E_1^{\times \rm c}$ as represented in~\eqref{E1timesc} does not have any hidden squares, so that
\begin{align*}
\Pcal_{E_1^{ \times \rm c}}\subset \bigcup_{i=1}^3\{A_i^{\rm co}\times A_i^{\rm co}\}\cup \{M^{\rm co}\times M^{\rm co}\} = \{B^{\rm co}\times B^{\rm co}: B\times B\in \Pcal_{E}\}.
\end{align*}	
In view of Lemma~\ref{lem:well-definition-Ewsc}, this implies $E^{\times \rm c}=E_1^{\times \rm c}$.\smallskip

b) Suppose now that $A_i$ are pairwise disjoint such that the convex hulls $A_i^{\rm co}$ have pairwise nonempty intersection
and satisfy  
$M\setminus A_i^{\rm co} \neq \emptyset$ for $i=1,2,3$ with $M= \bigcup_{i, j=1,2,3, i\neq j} A_i^{\rm co}\cap A_j^{\rm co}$. 

Then, $\Pcal_E =\bigcup_{i=1}^3 \{A_i\times A_i\}$ by Conclusion~\ref{rem:hidden-squares}\,a) gives
\begin{align*}
E_1^{\times \rm c}=\bigcup_{i=1}^3  A_i^{\rm co}\times A_i^{\rm co},
\end{align*}
and we infer from  Conclusion~\ref{rem:hidden-squares}\,c)
 that $E_1^{\times \rm c}$ contains exactly four maximal Cartesian squares, namely,
$\Pcal_{E_1^{\times \rm c}}=\bigcup_{i=1}^3\{A_i^{\rm co}\times A_i^{\rm co}\}\cup \{M\times M\}$. 
Hence,
\begin{align*}
E_2^{\times \rm c}=E_1^{\times \rm c} \cup (M^{\rm co}\times M^{\rm co}).
\end{align*}
Since $E^{\times \rm c}\supset E_2^{\times \rm c}\supsetneq E_1^{\times \rm c}$, we conclude that $E_1^{\times \rm c}$ cannot be Cartesian convex.
In fact, it holds that $E^{\times \rm c}=E^{\times \rm c}_2$, which we can see by applying a) to the convexifications of the sets $A_i$, see Figure~\ref{fig1}.\color{black}
\end{example}

		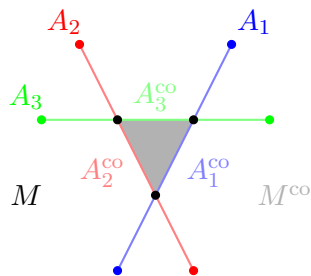
\begin{figure}[h]
		\begin{tikzpicture}
		\begin{scope}[xshift=-3cm]
		\filldraw[black!30!white] (-0.5,1)--(0,0)-- (0.5,1)-- (-0.5, 1);
		\draw[color=blue!50!white, thick](1,2) -- (-0.5,-1);
		\draw[color=red!50!white, thick](-1,2) -- (0.5, -1);
		\draw[color=green!50!white, thick](1.5,1) -- (-1.5,1);
		
		\filldraw [blue] (1,2) circle (1.5pt);
		\filldraw [blue] (-0.5,-1) circle (1.5pt);
		\filldraw [red] (-1,2) circle (1.5pt);
		\filldraw [red] (0.5,-1) circle (1.5pt);
		\filldraw [green] (1.5,1) circle (1.5pt);
		\filldraw [green] (-1.5,1) circle (1.5pt);
		
		\filldraw [black] (0.5,1) circle (1.5pt);
		\filldraw [black] (-0.5, 1) circle (1.5pt);
		\filldraw [black] (0,0) circle (1.5pt);
		
		\draw [black] (0.7,0.3) node {\color{blue!50!white} $A_1^{\rm co}$};
		\draw [black] (-0.7,0.3) node {\color{red!50!white} $A_2^{\rm co}$};
		\draw [black] (0, 1.3) node {\color{green!50!white} $A_3^{\rm co}$};
		\draw [black] (1.3,2.3) node {\color{blue} $A_1$};
		\draw [black] (-1.2,2.3) node {\color{red} $A_2$};
		\draw [black] (-1.7, 1.3) node {\color{green} $A_3$};
		
		\draw [black] (-1.7, 0) node {\color{black} $M$};
		\draw [black] (1.7, 0) node {\color{black!30!white} $M^{\rm co}$};
		\end{scope}	
		
		\end{tikzpicture}
		\caption{Illustration of Example~\ref{ex3}\,b) for $m=2$. 
		}
		\label{fig1}
	\end{figure}

It is well-known that the convex hull of a compact set is again compact. However, this is no longer true for the separately convex hull, see \cite[Example~2.2]{Aumann-Hart} and \cite[Example~2.1]{Kolar} (or \cite[Remark 8.7]{Dac08}).
Whether Cartesian convexification preserves compactness is currently an open question. 
Nevertheless, the following result holds. 
 
\color{black}
\begin{lemma}
\label{lem:compactness}
	Let $K\subset \R^m\times \R^m$ be symmetric, diagonal, and compact. Then, $K_k^{\times \rm c}$ is compact for any $k\in \N$.
\end{lemma}

\begin{proof} We prove the statement first for $k=1$.
Since $K_1^{\times \rm c}\subset K^{\rm co}$, it is clear that $K_1^{\times \rm c}$ is bounded. 
To see that it is also closed,
let $(\xi_j, \zeta_j)\subset K_1^{\times \rm c}$ for $j\in \N$ and suppose that $(\xi_j, \zeta_j)\to (\xi, \zeta)$ as $j\to \infty$ for some $(\xi, \zeta)\in \R^m\times \R^m$. We aim to show that $(\xi, \zeta)\in K_1^{\times \rm c}$.

In light of~\eqref{orderone}, there exists a sequence of compact 
sets $(B_j)_j\subset \R^m$ such that $B_j\times B_j\in \Pcal_K$ and $(\xi_j, \zeta_j)\in B_j^{\rm co}\times B_j^{\rm co}$ for all $j$.
From Lemma~\ref{lem:Hausdorffconverg}, we can infer the existence of a compact $B\subset \R^m$ with $B\times B\subset K$ and
\begin{align*}
d_H(B_j^{\rm co}\times B_j^{\rm co}, B^{\rm co}\times B^{\rm co})\leq d_H(B_j^{\rm co}, B^{\rm co})\to 0 \quad\text{as $j\to \infty$,}
\end{align*}
if necessary, after passing to a (non-relabeled) subsequence. Hence,
\begin{align*}
0& = \lim_{j\to \infty} d_H(B_j^{\rm co}\times B_j^{\rm co}, B^{\rm co}\times B^{\rm co})\\ 
& \geq \lim_{j\to \infty} \dist((\xi_j, \zeta_j), B^{\rm co}\times B^{\rm co}) = \dist((\xi, \zeta), B^{\rm co}\times B^{\rm co}),
\end{align*}
 which implies $(\xi, \zeta)\in B^{\rm co}\times B^{\rm co}\subset K_1^{\times \rm c}$, recalling that $B\times B$ is a subset of $K$.  \color{black}

The compactness of $K_k^{\times c}$ with $k\in \N$ can be obtained inductively by repeating the above argument. 
\end{proof}

\begin{remark} In~\eqref{convexhull}, we state the well-known characterization of convex hulls of compact sets via barycenters of probability measures. 
	An analogous representation formula holds for 
	$K_1^{\times \rm c}$ with $K\subset \R^m\times \R^m$ symmetric, diagonal, and compact. Precisely, 
\begin{equation*}
K^{\times \rm c}_1 = \{ [\Lambda] \colon \Lambda= \nu\otimes\mu \text{ with } \nu,\mu \in \Pcal r(\R^m) \text{ such that } \supp (\nu\otimes \mu) \subset B\times B \text{ for some } B\times B\in \Pcal_K \}.
\end{equation*}
This is straightforward to verify, given~\eqref{convexhull} and the observation that $B$ is compact for every $B\times B\in \Pcal_K$. 
\end{remark}
\color{black}

Next, we introduce a condition regarding the Cartesian convexification of symmetric and diagonal sets that turns out to be the key to characterizing structure preservation of both nonlocal inclusions under $L^\infty$-weakly$^\ast$ closure and nonlocal supremals under relaxation in the vectorial setting, cf.~Propositions~\ref{prop:structure_inclusions} and~\ref{prop:structure-preservation}. 

\color{black}
\begin{definition}[Basic Cartesian convexification]\label{def:condition(S)}
Let $E\subset\R^m\times \R^m$ be symmetric and diagonal. We say that $E$ has a basic Cartesian convexification, or a basic Cartesian convex hull, if 
\begin{align}\label{basic_convexification}
\Pcal_{E^{\times \rm c}}  \subset  \{B^{\rm co}\times B^{\rm co}: B\times B\in \Pcal_E\}. 
  \end{align}
\end{definition}
Let us make a few basic observations and discuss some useful properties of this notion. 
\begin{remark}\label{rem:NS}
Let $E\subset \R^m\times \R^m$ be symmetric and diagonal.
\smallskip

 a) Any Cartesian convex set $E$ has a basic Cartesian convexification. This is due to
  \begin{align*}
  \Pcal_{E^{\times \rm c}} =\Pcal_E 
 =\{B^{\rm co}\times B^{\rm co}:B\times B\in \Pcal_E\},
 \end{align*} where the last identity exploits the maximality property of the elements of $\Pcal_E$ as well as~Lemma~\ref{lem:alternative}. 

\smallskip

b) Suppose that $E$ has a basic Cartesian convexification, then 
\begin{equation}\label{N}
E^{\times \rm c} = E^{\times \rm c}_1.
\end{equation}
Indeed, we deduce from
~\eqref{symdiag} and~\eqref{basic_convexification} that
$$
E^{\times \rm c}=\bigcup_{A\times A\in \Pcal_{E^{\times \rm c}}} A\times A \subset \bigcup_{B\times B\in \Pcal_E} B^{\rm co}\times B^{\rm co} = E_1^{\times\rm c} \subset E^{\times \rm c}.
$$ 
If $E$ is additionally compact, then $E^{\times \rm c}$ is also compact by~Lemma~\ref{lem:compactness}. 

\smallskip

c) Note that~\eqref{basic_convexification} can be replaced by 
\begin{equation}\label{basic_convexification_1}
\Pcal_{E^{\times \rm c}_1}  \subset  \{B^{\rm co}\times B^{\rm co}: B\times B\in \Pcal_E\}.
\end{equation}
To see this, recall  Lemma~\ref{lem:well-definition-Ewsc} and apply~\eqref{basic_convexification_1} iteratively to obtain $E^{\times \rm c}_1=E^{\times \rm c}_2 =\ldots = E^{\times \rm c}_k$ for every $k\in\N$, and thus, $E^{\times \rm c} = E^{\times \rm c}_1$. 

\smallskip

d) In the scalar setting and for compact sets, the condition of having a basic Cartesian convexification is no restriction.
If $E\subset \R\times \R$ is compact, we know $E^{\times \rm c}= E^{\rm sc}$ by Remark~\ref{rem:comparison}\,a), and therefore,  due to \cite[Lemma~4.7]{KrZ19}, 
\begin{align*}
\Pcal_{E^{\times \rm c}} =\Pcal_{E^{\rm sc}} &\subset \{[\alpha, \beta]\times [\alpha,\beta] \colon (\alpha,\beta)\in E\} \subset \{ B^{\rm co} \times B^{\rm co} \colon B\times B \subset E \}.
\end{align*} 
 In particular, it follows then from b) that $E^{\times \rm c} = E^{\times \rm c}_1$.  
\smallskip

\color{black}

e) An example of a set that fails to satisfy the necessary condition~\eqref{N} and does therefore not possess a basic Cartesian convexification, 
has been discussed already in Example~\ref{ex3}\,b).
\end{remark}

Finally, let us comment on the relation between $E^{\times \rm c}$ and $\widehat{E^{\rm sc}}$.
Even though we know that $\widehat{E^{\rm sc}}$ is Cartesian convex by Lemma~\ref{lem:alternative} and clearly, $E^{\times \rm c} \subset \widehat{E^{\rm sc}}$, it is currently still open whether the identity 
\begin{equation*}
E^{\times \rm c} = \widehat{E^{\rm sc}}
\end{equation*}
holds in general in the vectorial case. For $m=1$, it follows immediately from Remark~\ref{rem:comparison}\,a), and it is also satisfied for the higher-dimensional example in Remark~\ref{rem:comparison}\,b). 
Besides, the question whether $\widehat{E^{\rm sc}}$ and $E^{\times \rm c}$ are always compact provided $E$ is compact remains to be resolved.

\subsection{Cartesian level convexity of functions}  Our new notion of convexity for sets can be carried over  to functions by requiring that their sublevel sets are Cartesian convex. 
 
\begin{definition}[Cartesian level convexity]\label{def:wslc}
A function $W:\R^m\times \R^m\to \R$ is called Cartesian level convex if all sublevel sets of $W$, i.e.,~$L_c(W)$ with $c\in \R$, are Cartesian convex.
\end{definition}

As we prove next, there is a (generalized) Jensen-type inequality for Cartesian level convex functions.

\begin{proposition}\label{prop:equiv-wslc}
	Let $W \colon \R^m \times \R^m \to \R$ be lower semicontinuous.
	Then, the following statements are equivalent:
	\begin{itemize}
		\item[$(i)$] $W$ is Cartesian level convex;\\[-0.3cm]
		\item[$(ii)$] for any $\xi_1, \xi_2, \zeta_1, \zeta_2 \in \R^m$, it holds that
		\begin{align}\label{char_wslc}
		\begin{split}
		W(\alpha, \beta) &\le  \max_{i,j \in \{1,2\} } \{W(\xi_i, \zeta_j), W(\zeta_i, \xi_j), W(\xi_i, \xi_j), W(\zeta_i, \zeta_j)\}  \\ & 
		= \max_{(\xi,\zeta)\in \{\xi_1, \xi_2, \zeta_1, \zeta_2\}\times \{\xi_1, \xi_2, \zeta_1, \zeta_2\}} W(\xi,\zeta) 
		\end{split}
		\end{align}
		for all $(\alpha, \beta)\in [\xi_1, \xi_2]\times [\zeta_1, \zeta_2]$;\\[-0.3cm]
		\item[$(iii)$]  for every $\nu, \mu \in \Pcal r(\R^m)$, 
		\begin{equation}\label{Jensen_xslc}
		W([\nu], [\mu]) \le \max \left\{ (\nu \otimes \mu)\text{-} \esssup  W,  (\mu \otimes \nu)\text{-} \esssup  W, (\nu \otimes \nu)\text{-} \esssup W, (\mu \otimes \mu)\text{-} \esssup W \right\}.
		\end{equation}
	\end{itemize}
\end{proposition}

\begin{proof} 
 $(i)\Rightarrow (ii)$ Let $\xi_1, \xi_2, \zeta_1, \zeta_2 \in \R^m$ and $(\alpha, \beta)\in [\xi_1, \xi_2]\times [\zeta_1, \zeta_2]$. Moreover, let $c$  be the right-hand side of~\eqref{char_wslc}.
Setting $A:=\{\xi_1, \xi_2, \zeta_1, \zeta_2\}$, it holds that
	\begin{align*}
	A\times A\subset L_{c}(W).
	\end{align*} 
It follows then by the Cartesian convexity of $L_{c}(W)$  
that
	\begin{align*}
	[\xi_1, \xi_2]\times [\zeta_1, \zeta_2]\subset A^{\rm co}\times A^{\rm co}\subset L_{c}(W),\end{align*}
which implies $W(\alpha, \beta)\leq c$.\smallskip

	$(ii)\Rightarrow (i)$ Let $c\in \R$ and $(\xi_1,\zeta_1), (\xi_2, \zeta_2)\in L_c(W)$ be such that $\{\xi_1, \xi_2, \zeta_1, \zeta_2\}\times \{\xi_1, \xi_2, \zeta_1, \zeta_2\}\subset L_c(W)$. We then infer from~\eqref{char_wslc} that $W(\alpha, \beta)\leq c$ for all 
 $(\alpha, \beta)\in [\xi_1, \xi_2]\times [\zeta_1, \zeta_2]$, and hence, $ [\xi_1, \xi_2]\times [\zeta_1, \zeta_2]\subset L_c(W)$.
 \smallskip

\color{black}
	$(iii)\Rightarrow (ii)$ If $\alpha=s\xi_1+(1-s)\xi_2$ and $\beta=t\zeta_1 + (1-t)\zeta_2$ for $s, t\in [0,1]$ and $\xi_1,\xi_2,\zeta_1,\zeta_2 \in \R^m$, then $(ii)$ follows directly from the choice of the probability measures 
	\begin{center}
	$\nu=s\delta_{\xi_1} + (1-s)\delta_{\xi_2}$\quad and\quad  $\mu=t\delta_{\zeta_1} + (1-t)\delta_{\zeta_2}$.
	\end{center}
	
 $(ii)\Rightarrow (iii)$ Let $\nu, \mu \in \Pcal r(\R^m)$ and denote the right-hand side of~\eqref{Jensen_xslc} by $c$.
Since $[\nu]\in (\supp \nu)^{\rm co}$ and $[\mu] \in (\supp \mu)^{\rm co}$, 
by Caratheodory's formula for convex hulls~\cite[Theorem 2.13]{Dac08}, there exist $\xi_i\in \supp \nu$, $ \zeta_j\in \supp \mu$ and $t_i, s_j  \in [0,1]$ for $i, j=1, \ldots, m+1$ such that $\sum_{i=1}^{m+1} t_i = 1 = \sum_{j=1}^{m+1} s_j$ and 
	$$
	([\nu], [\mu]) = \bigl( \textstyle\sum_{i=1}^{m+1} t_i \xi_i \, , \, \sum_{j=1}^{m+1} s_j \zeta_j   \bigr).
	$$

We will show next that 
\begin{align}\label{AtimesALcW}
A\times A\subset L_c(W) \qquad \text{with $A=\{ \xi_1,\ldots, \xi_{m+1}, \zeta_1, \ldots, \zeta_{m+1}\}.$}
\end{align}
Indeed, for all $i,j$, one obtains in view of $(\xi_i,\zeta_j) \in \supp \nu \times \supp \mu$ and the identity~\eqref{ess-supp} that
\begin{align*}
 W(\xi_i,\zeta_j)\le \sup_{\supp (\nu\otimes \mu)} W =(\nu\otimes\mu)\text{-}\esssup W \leq c, 
\end{align*} 
and therefore $(\xi_i, \zeta_j)\in L_c(W)$; an analogous argument shows that $(\zeta_i, \xi_j), (\xi_i,\xi_j), (\zeta_i, \zeta_j)\in L_c(W)$. This implies~\eqref{AtimesALcW}. 
	
Since the sublevel set $L_c(W)$ is Cartesian convex, 
we can invoke Lemma~\ref{lem:alternative} to conclude that 
$$
([\nu], [\mu]) \in A^{\rm co}\times A^{\rm co} \subset L_c(W).
$$
\end{proof}

In parallel to Section~\ref{subsec:Cconvexsets}, let us comment briefly on the relation between Cartesian and separate level convexity. Recall that a function $W:\R^m\times \R^m\to \R$ is separately level convex if all sublevel sets of $W$ are separately convex. \color{black}

\begin{remark} [Comparison with separate level convexity]\label{rem:slc}
The equivalence of these three conditions for a lower semicontinuous $W:\R^m\times \R^m\to \R$ is known from~\cite[Lemma~3.5]{KrZ19}:
	\begin{itemize}
		\item[$(i)$] $W$ is separately level convex;\\[-0.3cm]
		\item[$(ii)$] for any $\xi_1, \xi_2, \zeta_1, \zeta_2 \in \R^m$, it holds that
		\begin{align}\label{Jensen_bsep}
		W(\alpha, \beta) \le  \max_{i,j \in \{1,2\} } W(\xi_i, \zeta_j) =
		\max_{(\xi,\zeta)\in \{\xi_1, \xi_2\}\times \{\zeta_1, \zeta_2\}} W(\xi,\zeta)
		\end{align}
		for all $(\alpha, \beta)\in [\xi_1, \xi_2]\times [\zeta_1, \zeta_2]$;\\[-0.3cm]
		\item[$(iii)$] for every $\nu, \mu \in \Pcal r(\R^m)$, 
		\begin{align}\label{Jensen_slc}
		W([\nu], [\mu]) \le  (\nu \otimes \mu)\text{-} \esssup W. 
		\end{align}
	\end{itemize} 
It is apparent that $(ii)$ and $(iii)$ imply Proposition~\ref{prop:equiv-wslc}\,$(ii)$ and $(iii)$, which reflects that any separately level convex function is also Cartesian level convex. 
	When $m=1$ and $W$ is also symmetric and diagonal, the notions of Cartesian and separate level convexity are identical in light of Remark~\ref{rem:comparison}\,a). Therefore, the Jensen's formulas~\eqref{char_wslc} and~\eqref{Jensen_bsep} as well as \eqref{Jensen_xslc} and \eqref{Jensen_slc} coincide in that case. 	\smallskip
\end{remark}

 When $W$ fails to be Cartesian level convex, the corresponding envelope defined by
\begin{align}\label{Wtimeslc}
W^{\rm \times lc}(\xi, \zeta) := \sup\{V(\xi, \zeta): V:\R^m\times \R^m\to \R \text{ is Cartesian level convex and $V\leq W$}\},
\end{align}
for $(\xi, \zeta)\in \R^m\times \R^m$ provides the largest Cartesian level convex function not exceeding $W$. The next lemma, which is based on the results of Section~\ref{sec:Cartesianconvexification_set}, implies under suitable assumptions on $W$ and under consideration of~\eqref{levelset-i} a natural representation of $W^{\times \rm lc}$, namely,
\begin{align*}
W^{\times \rm lc}(\xi, \zeta)= \inf \{c\in \R: (\xi, \zeta)\in {L_c(W)}^{\times\rm c}\}\qquad  \text{for $(\xi,\zeta)\in \R^m\times \R^m$. }
\end{align*}

\begin{lemma}\label{lem:W-wslc} 
 Let $W:\R^m\times \R^m\to \R$ be symmetric, diagonal, lower semicontinuous and coercive.
 Then, $W^{\times \rm lc}$ is also symmetric, diagonal, and coercive, 
 and satisfies for any $c\in \R$ that
\begin{equation}\label{eq:level-sets-wsc}
L_c(W^{\times \rm lc}) = L_c(W)^{\times \rm c}.
\end{equation}

If all sublevel sets of $W$ admit a basic Cartesian convexification in the sense of Definition~\ref{def:condition(S)}, then $W^{\times \rm lc}$ is also  lower semicontinuous. 
 \color{black}
\end{lemma}

\begin{proof} 
To show~\eqref{eq:level-sets-wsc}, let $c\in \R$ be fixed. 
Since $W^{\times \rm lc} \le W$, and thus, $L_c(W) \subset L_c(W^{\times \rm lc})$, one has
\begin{align*}
L_c(W)^{\times \rm c}\subset L_c(W^{\times \rm lc})^{\times \rm c} =L_c(W^{\times \rm lc}),
\end{align*} 
where the last identity follows from the Cartesian level convexity of $W^{\times \rm lc}$.

For the reverse inclusion, consider the auxiliary function
\begin{equation*} 
\widetilde{W}(\xi,\zeta) := \inf \{c\in \R: (\xi, \zeta)\in {L_c(W)}^{\times\rm c}\} \quad \text{ for $(\xi,\zeta)\in \R^m\times \R^m$}.
\end{equation*}
Since Lemma~\ref{lemma:basic-topo} gives
	\begin{equation*}
	\bigcap_{j\in \N}  L_{c+\frac{1}{j}}(W)^{\times\rm c} =	L_{c}(W)^{\times \rm c},
	\end{equation*}
	 it follows with~\eqref{level-sets-ii} that $L_c(\widetilde{W}) = {L_c(W)}^{\times \rm c}$.  Consequently, $\widetilde{W}$ is Cartesian level convex and therefore, $W^{\times \rm lc}\geq \widetilde{W}$. This implies
$$
L_c(W^{\times \rm lc}) \subset L_c( \widetilde{W}) = {L_c(W)}^{\times \rm c},
$$
completing the proof of~\eqref{eq:level-sets-wsc}. 

The coercivity and lower semicontinuity of $W^{\times \rm lc}$ results from the observation that its level sets are bounded and closed, respectively. The latter holds under the additional assumption of  a basic Cartesian convex hull due to~Remark~\ref{rem:NS}\,b) and Lemma~\ref{lem:compactness}. 
\color{black}
\end{proof}

Finally, we relate $W^{\times \rm lc}$ with the separately level convex envelope of $W$, defined in analogy to~\eqref{Wtimeslc} and denoted by $W^{\rm slc}$ . 

\begin{remark} [Comparison with separate level convexification]\label{rem:slc_hull}
Let $W:\R^m\times \R^m\to \R$ be symmetric, diagonal, and lower semicontinuous. 
Then, $W^{\times \rm lc}\geq \widehat{W^{\rm slc}}$ 
is a consequence of  Remark~\ref{rem:slc} along with the symmetry and diagonality of  $W^{\times \rm lc}$ (cf.~Lemma~\ref{lem:W-wslc}). The validity of the identity 
$$W^{\times \rm lc}=\widehat{W^{\rm slc}}$$ remains an open question, but we can show that $\widehat{W^{\rm slc}}$ is Cartesian level convex by the generalized Jensen's inequality in Proposition~\ref{prop:equiv-wslc}\,$(iii)$. Indeed, for any $\nu, \mu \in \Pcal r(\R^m)$,
\begin{align*}
&\max \bigl\{ (\nu \otimes \mu)\text{-} \esssup \widehat{W^{\rm slc}}, (\mu \otimes \nu)\text{-} \esssup \widehat{W^{\rm slc}}, (\nu \otimes \nu)\text{-} \esssup \widehat{W^{\rm slc}}, (\mu\otimes\mu)\text{-}\esssup  \widehat{W^{\rm slc}}\bigr\} \\ 
&\quad\geq \max \left\{ (\nu \otimes \mu)\text{-} \esssup W^{\rm slc}, (\mu \otimes \nu)\text{-} \esssup W^{\rm slc}, (\nu \otimes \nu)\text{-} \esssup W^{\rm slc},  (\mu \otimes \mu)\text{-} \esssup W^{\rm slc}\right\} \\
&\quad \geq \max\{W^{\rm slc}([\nu], [\mu]), W^{\rm slc}([\mu], [\nu]), W^{\rm slc}([\nu], [\nu]), W^{\rm slc}([\mu], [\mu]), \} \\ 
&\quad =\widehat{W^{\rm slc}}([\nu], [\mu]),
\end{align*} 
where we have exploited~\eqref{eq:W-hat} along with the observations of Remark~\ref{rem:slc}.
\end{remark}

\section{Weak$^\ast$ closures of nonlocal inclusions}\label{sec:inclusions}
The relaxation of nonlocal $L^\infty$-functionals is closely interlinked with the characterization of weak$^\ast$ limits of sequences satisfying nonlocal inclusions. 
Building on previous findings in \cite{KrZ19, Kreisbeck-Zappale-2019Loss}, we approach this problem from a
different angle and investigate the issue of structure preservation 
in the general vectorial setting.

For $E\subset \R^m\times\R^m$, we introduce the set of solutions $u\in L^{\infty}(\Omega;\R^m)$ to the nonlocal inclusion $(u(x), u(y))\in E$  for a.e. $(x,y) \in \Omega\times \Omega$, 
\begin{equation*}
\mathcal{A}_E= \left\{  u\in L^{\infty}(\Omega;\R^m) \colon (u(x),u(y)) \in E \;\text{  for a.e. } (x,y)\in \Omega\times \Omega  \right\}.
\end{equation*}
In~\cite{KrZ19}, for a compact set $K\subset \R^m\times\R^m$, it was proved that 
\begin{equation}\label{AK-alt}
\Acal_{K} = \bigcup_{B\times B\in \Pcal_K} \Acal_{B\times B} = \bigcup_{B\times B\in \Pcal_K} L^\infty(\Omega;B),
\end{equation}
where $L^\infty(\Omega;B)$ denotes the essentially bounded Lebesgue measurable functions from $\Omega$, taking values in $B$. 
Moreover,
the inclusion $\Acal_K$ is invariant under symmetrization and diagonalization of $K$, i.e.,
\begin{equation*}
\Acal_K=\Acal_{\widehat K}, 
\end{equation*}
and that the order principle
\begin{align}\label{order_principle}
\Acal_K\subset \Acal_L\qquad \text{ if and only if}\qquad  \widehat K\subset\widehat L
\end{align}
 holds for all $K, L\subset \R^m\times \R^m$ compact. 
 
 We define the weak$^\ast$ closure of $\Acal_E$ as
\begin{align*}
\mathcal{A}_E^{\infty}:= \left\{  u\in L^{\infty}(\Omega;\R^m) \colon   u_j\rightharpoonup^* u \text{ in } L^{\infty}(\Omega;\R^m) \text{ with } (u_j)_j \subset \mathcal{A}_E\right\}
\end{align*}
in analogy to~\cite{KrZ19}. 
If $K\subset \R^m\times \R^m$ is compact, then this weak$^\ast$ closure can be expressed as
	\begin{align}\label{characterization_inclusionAcal1}
	\mathcal{A}^{\infty}_K=\bigcup_{B\times B\in \Pcal_K} \Acal_{B^{\rm co}\times B^{\rm co}} = \bigcup_{B\times B\in \Pcal_K} L^\infty(\Omega;B^{\rm co}),
\end{align}
see~\cite[Remark 3.2 b)]{Kreisbeck-Zappale-2019Loss} and \cite[Theorem 1.1]{KrZ19}. 
We point out the strong structural resemblence between~\eqref{characterization_inclusionAcal1} and the representation formula for $\Acal_K$ in~\eqref{AK-alt}. In both cases, it suffices to consider the nonlocal inclusions restricted to the maximal Cartesian sets of $K$, whose solutions trivially correspond to essentially bounded functions with a suitably restricted codomain. \smallskip

Our guiding question in this section is the following:
\begin{center}
Under what assumptions on $K$ does the weak$^\ast$ closure of $\Acal_K$ preserve the structure of the nonlocal inclusion, that is, when is $\Acal_K^\infty=\Acal_L$ with a suitable $L\subset \R^m\times \R^m$?
\end{center}

According to~\cite[Theorem 1.1]{KrZ19}, structure preservation of $\Acal_K$ under weak$^\ast$ closures is always guaranteed when $m=1$, and it was shown to hold under specific assumptions in the vectorial case. 
As the first step towards a complete picture for arbitrary dimensions, we start by identifying the Cartesian convexity of $K$ as a necessary and sufficient condition for $\mathcal{A}_K$ being weakly$^*$ closed in $L^{\infty}(\Omega;\R^m)$, or in other words, for $\mathcal{A}_K^{\infty}= \mathcal{A}_{K}$.\\

\begin{proposition}~\label{prop:weakstar_closure}
Let $K\subset \R^m\times \R^m$ be symmetric, diagonal, and compact. 
Then, $\Acal_K$ is $L^\infty$-weakly$^\ast$ closed if and only if $K$ is Cartesian convex.
\end{proposition}

\begin{proof}
To prove the sufficiency, observe that the Cartesian convexity of $K$ implies 
\begin{align*}
K=\bigcup_{B\times B\in \Pcal_K} B^{\rm co}\times B^{\rm co}
\end{align*} 
by Lemma~\ref{lem:alternative}\,$(iv)$. Therefore, along with~\eqref{characterization_inclusionAcal1},
\begin{align*}
\Acal_K\subset \Acal_K^\infty = \bigcup_{B\times B\in \Pcal_K} \Acal_{B^{\rm co}\times B^{\rm co}}\subset \Acal_{\bigcup_{B\times B\in \Pcal_K} B^{\rm co}\times B^{\rm co}} = \Acal_K.
\end{align*}
As this chain of inclusions turns into identities, the weak$^\ast$ closedness of $\Acal_K$ follows. \color{olive}

\color{black} For the necessity, suppose that $\Acal_K^\infty=\Acal_K$. Then, again by~\eqref{characterization_inclusionAcal1}, it holds for any $B\times B\in \mathcal P_K$ that  $\Acal_{B^{\rm co}\times B^{\rm co}}\subset \Acal_K^{\infty}=\Acal_K$, so that $B^{\rm co}\times B^{\rm co}\subset K$ in view of~\eqref{order_principle}. 
\color{black} Consequently, 
\begin{align*}
K=\bigcup_{B\times B\in \mathcal P_K} B^{\rm co}\times B^{\rm co},
\end{align*}
which, according to Lemma~\ref{lem:alternative}, shows that $K$ is Cartesian  convex.
\end{proof}

The previous statement provides a natural candidate for the compact set $L\subset \R^m\times \R^m$ in the representation $\Acal_K^\infty=\Acal_L$ (if existent), namely, 
\begin{align}\label{465}
L=K^{\times \rm c}. 
\end{align} 
Indeed, since $L$ needs to be Cartesian convex by Proposition~\ref{prop:weakstar_closure} and contains $K$ because of $\Acal_K\subset\Acal_K^\infty=\Acal_L$ in view of~\eqref{order_principle}, clearly, $L\supset K^{\times \rm c}$. On the other hand, it follows from~\eqref{characterization_inclusionAcal1} that
\begin{align*}
\Acal_L=\Acal_{K}^\infty = \bigcup_{B\times B\in \Pcal_K}\Acal_{B^{\rm co}\times B^{\rm co}} \subset \Acal_{K^{\times \rm c}_1},
\end{align*} 
which implies $L\subset K^{\times \rm c}_1\subset K^{\times \rm c}$ in view of~\eqref{order_principle}. 
Notice that $K_1^{\times \rm c}$ is, besides being symmetric and diagonal, also compact according to~Lemma~\ref{lem:compactness}. 

However, $\Acal_{K}^\infty=\Acal_{K^{\times \rm c}}$ cannot be true for all $K$; we have implicitly encountered a limitation already in the previous reasoning, which requires 
\begin{align*}
K^{\times \rm c} = K_1^{\times \rm c}.
\end{align*} 
Recalling that the Cartesian convex hull of $K$ is called basic in the sense of Definition~\ref{def:condition(S)} if $\Pcal_{K^{\times \rm c}} \subset \{B^{\rm co}\times B^{\rm co}: B\times B\in \Pcal_K\}$ allows us to deduce the following characterization result.

\begin{proposition}\label{prop:structure_inclusions}
Let $K\subset \R^m\times \R^m$ be symmetric, diagonal, and compact. Then,
\begin{align*}
 \Acal_{K}^\infty=\Acal_{K^{\times \rm c}}\qquad \text{if and only if} \qquad \text{$K$ has a  basic Cartesian convexification}.
\end{align*}
\end{proposition}
The proof is 
 a direct consequence of the next lemma.

\begin{lemma}\label{lem:aux1}
Let $K, L\subset \R^m\times \R^m$ be symmetric, diagonal, and compact. Then, $\Acal_K^\infty=\Acal_L$ if and only if 
\begin{align}\label{condition1}
L=
K^{\times \rm c} \quad \text{and} \quad 
\mathcal P_{L} \subset   \{B^{\rm co}\times B^{\rm co}: B\times B \in \mathcal P_K\}.
\end{align}
\end{lemma}

\begin{proof} 
 If $\Acal_K^\infty=\Acal_{L}$, then
$L= K^{\times \rm c}$ by~\eqref{465}. To show also the second part of~\eqref{condition1}, let $D\subset \R^m$ with $D\times D\in \Pcal_{L}$. 
Moreover, take $D'=\{d_1,d_2, \ldots\}$ to be a countable dense subset of $D$, e.g.~$D'=D\cap \Qbb^m$, and set $D_i:=\{d_1, \ldots, d_i\}$ for $i\in \N$; note that these sets are nested, that is, $D_{i}\subset D_{i+1}$ for $i\in \N$, and satisfy
	$D'= \bigcup_{i\in \N} D_i$. 
	
If we consider for each $i\in \N$ a simple function $u_i:\Omega\to D_i$ such that each preimage $u_i^{-1}(d_j)$ for $j=1, \ldots, i$ has positive Lebesgue measure, then, by construction, $u_i\in \Acal_L =\Acal_K^\infty$ 
and thus, by~\eqref{characterization_inclusionAcal1}, $u_i\in L^\infty(\Omega;B_i^{\rm co})$ for some  $B_i\times B_i \in \Pcal_K$.

Applying Lemma~\ref{lem:Hausdorffconverg} to the sequence of compact sets $(B_i)_i$ gives a compact set $B\subset \R^m$ with $B\times B\subset K$ such that  
\begin{align*}
\sup_{\xi\in D'}\dist (\xi,B^{\rm co}) = \lim_{i\to \infty} \sup_{\xi\in D_i} \dist(\xi, B^{\rm co}) \leq \liminf_{i\to \infty} d_H(B_i^{\rm co}, B^{\rm co}) = 0. 
\end{align*}
Hence, $D'\subset B^{\rm co}$, and taking the closure implies $D\subset B^{\rm co}$ and 
$$
D\times D \subset B^{\rm co}\times B^{\rm co} \subset K^{\times \rm c}.
$$ 
After enlarging $B$, we may assume that  $B\times B\in \Pcal_K$,
and it follows due to the maximality of $D\times D$ in $L= K^{\times \rm c}$ that $D=B^{\rm co}$. 
\smallskip

Now, suppose that~\eqref{condition1} holds. 
Since $L=K^{\times \rm c}$ is Cartesian convex, we infer from  Proposition~\ref{prop:weakstar_closure} that 
$\Acal_L^\infty=\Acal_L$. 
It follows then together with~\eqref{AK-alt},~\eqref{characterization_inclusionAcal1}, and  $K\subset K^{\times \rm c}=L$ that
\begin{align*}
\Acal_L= \bigcup_{D\times D\in \Pcal_L} L^{\infty}(\Omega;D) \subset \bigcup_{B\times B\in \Pcal_K} L^{\infty}(\Omega;B^{\rm co})=\Acal_K^{\infty} \subset \Acal_L^\infty = \Acal_L,
\end{align*}
which yields $\Acal_K^{\infty}=\Acal_L$.   
\end{proof}

Observe that 
the previous findings are
 consistent with~\cite[Theorem~1.1]{KrZ19}, which is actually a special case of Proposition~\ref{prop:structure_inclusions}. Suppose $K\subset \R^m\times \R^m$ is symmetric, diagonal, and compact such that $K^{\rm sc}$ is compact and satisfies
 \begin{align}\label{cond90}
\widehat{K^{\rm sc}} = \bigcup_{(\alpha, \beta)\in K} [\alpha, \beta]\times [\alpha, \beta] \quad \text{and}\quad \Pcal_{\widehat{K^{\rm sc}}}= \Pcal_{K^{\rm sc}} \subset \{[\alpha, \beta]\times [\alpha, \beta]: \{\alpha, \beta\}\times \{\alpha, \beta\}\in K\}.
 \end{align}
Under these assumptions, one can mimic the one-dimensional argument of~Remark~\ref{rem:comparison}\,a) to derive the inclusion $\widehat{K^{\rm sc}}\subset K^{\times \rm c}_1$. Together with $K^{\times \rm c}\subset \widehat{K^{\rm sc}}$, it follows that
\begin{align}\label{98-}
\widehat{K^{\rm sc}} = K_1^{\times \rm c} = K^{\times \rm c}.
\end{align}
Since the second part of~\eqref{cond90} in combination with~\eqref{98-} implies that  the Cartesian convex hull of $K$ is basic, we can apply 
 Proposition~\ref{prop:structure_inclusions} to infer $\Acal_K^\infty = \Acal_{K_1^{\times \rm c}}$, and hence, 
	\begin{align*}
	\Acal_K^\infty = \Acal_{\widehat{K^{\rm sc}}}= \Acal_{K^{\rm sc}},
	\end{align*}
	as stated in \cite[Theorem~1.1]{KrZ19}.\smallskip

We conclude our discussion of nonlocal inclusions with an alternative representation of $\Acal_K^\infty$ via barycenters of Young measures. It serves as a useful technical tool in the final section of this paper on $L^p$-approximation.

\begin{remark}[Young measure representation of \boldmath{$\Acal_K^\infty$}]\label{rem:YM_representation}  
For any $K\subset \R^m\times \R^m$ symmetric and compact, let 
\begin{align}\label{Vcal_K}
\Vcal_{K} &:= \{\nu\in L_w^\infty(\Omega;\mathcal Pr(\R^m)): \supp (\nu_x\otimes \nu_y) \subset K \text{ for a.e.~$(x,y)\in \Omega\times \Omega$}\}.\end{align}
We claim that 
\begin{align}\label{YMAkinfty}
\Acal_K^\infty = \{u\in L^\infty(\Omega;\R^m): u=[\nu], \nu\in \Vcal_K\},
\end{align}
or in other words,
\begin{align}\label{equivalenceVcal}
\nu\in \Vcal_K \qquad\text{if and only if} \qquad [\nu]\in \Acal_K^\infty. 
\end{align}
This characterization of $\Acal_{K}^\infty$ is essentially a reformulation of the results in~\cite[Theorems 3.3 and 3.1]{Kreisbeck-Zappale-2019Loss}.  
A sketch of the corner points of the derivation is included below for the reader's convenience; for further details, see~\cite[Section 3]{Kreisbeck-Zappale-2019Loss}. As for notation, we associate with every $u\in L^{\infty}(\Omega;\R^m)$, the nonlocal vector field
\begin{align}\label{eq:v_u}
v_u(x, y):= (u(x), u(y)) \qquad\text{ for $(x,y) \in \Omega\times \Omega$. }
\end{align}

The principal identity behind~\eqref{YMAkinfty} and~\eqref{equivalenceVcal} is 
\begin{align}\label{Ycal}
\Ycal_K^\infty=\Ycal_K
\end{align} 
from \cite[Theorem 3.3]{Kreisbeck-Zappale-2019Loss}, where
\begin{align*}
\Ycal_K^{\infty}=\{ \Lambda \in L_w^\infty(\Omega\times \Omega;\mathcal Pr(\R^m\times \R^m))\colon & v_{u_j}\stackrel{\rm YM\ }{\longrightarrow}\Lambda, \text{ with } (u_j)_j \subset \Acal_K\}
\end{align*}
and
\begin{align*}
\Ycal_K&=\{ \Lambda \in L_w^\infty(\Omega\times \Omega;\mathcal Pr(\R^m\times \R^m))\colon \Lambda_{(x,y)}= \nu_x\otimes \nu_y \text{ with } \nu \in  L_w^\infty(\Omega;\mathcal Pr(\R^m)),\\
&\hspace{5.6cm} \supp (\nu_x\otimes \nu_y) \subset K \text{ a.e.~$(x,y)\in \Omega\times \Omega$} \}\\ 
&=\{ \Lambda \in L_w^\infty(\Omega\times \Omega;\mathcal Pr(\R^m\times \R^m))\colon \Lambda_{(x,y)}= \nu_x\otimes \nu_y \text{ with } \nu \in  \Vcal_K\}.
\end{align*}
The proof of~\eqref{Ycal} relies on characterization results of Young measures generated by sequences of both exact and approximate solutions to the nonlocal inclusion problem, together with an approximating argument for elements in $\Ycal_K$ by suitable inhomogeneous convex combinations of Dirac masses. 

By taking barycenters, we obtain with the same reasoning as in the proof of~\cite[Theorem 3.1]{Kreisbeck-Zappale-2019Loss} that
\begin{align*}
\Acal_K^{\infty}&=\{ u\in L^{\infty}(\Omega;\R^m) \colon v_u = [\Lambda], \, \Lambda \in \Ycal_K^{\infty} \}=\{ u\in L^{\infty}(\Omega;\R^m) \colon v_u = [\Lambda], \, \Lambda \in \Ycal_K\}\\
&=\{ u\in L^{\infty}(\Omega;\R^m) \colon u= [\nu], \, \nu \in \Vcal_K\},
\end{align*}
as stated~\eqref{YMAkinfty}.
\end{remark}

\section{Lower semicontinuity and relaxation of nonlocal supremals}\label{sec:lower-rlx}

In this section, we consider nonlocal supremal functionals of the form $J_W$ as defined in \eqref{es}, with a lower semicontinuous and coercive supremand $W:\R^m\times \R^m\to \R$. 
Since $J_{W}= J_{\widehat{W}}$, that is, 
\begin{equation}\label{eq:J=Jhat}
\esssup_{(x, y)\in \Omega\times \Omega} W(u(x), u(y))=  \esssup_{(x, y)\in \Omega\times \Omega} \widehat{W}(u(x), u(y))
\end{equation}
for every $u\in L^{\infty}(\Omega;\R^m)$ (see~\cite[(7.3)]{KrZ19}), we may assume without loss of generality that $W$ is also symmetric and diagonal. 

With Proposition~\ref{prop:weakstar_closure} from the previous section at hand, we are now in the position to prove Theorem~\ref{theo:charact-wls}, which identifies Cartesian level convexity of the supremand as necessary and sufficient for the weak$^\ast$ lower semicontinuity of nonlocal $L^\infty$-functionals. 
Considering that the notions of Cartesian level convexity and separate level convexity coincide in the scalar case (cf.~Remark~\ref{rem:comparison}\,a)), this generalizes the findings of~\cite[Theorem 1.3\,$(i)$]{KrZ19} to arbitrary dimensions. 

\begin{proof}[Proof of Theorem \ref{theo:charact-wls}] By~\cite[Proposition~7.1]{KrZ19}, the $L^\infty$-weak$^\ast$  lower semicontinuity of $J_W$ is equivalent to the $L^\infty$-weak$^\ast$ closedness of $\Acal_{L_c(W)}$ for any $c\in \R$. Since all sublevel sets of $W$ are symmetric, diagonal, and compact in light of the assumptions on $W$, the stated equivalence is an immediate consequence of Proposition~\ref{prop:weakstar_closure}.
\end{proof}

When $J_{W}$ fails to be weakly$^*$ lower semicontinuous, finding a good representation of its weakly$^*$ lower semicontinuous envelope (see~\eqref{JW_rlx}) is a central question in relaxation theory. 
Our next auxiliary result establishes the connection between the relaxation of nonlocal supremal functionals and the $L^\infty$-weak$^\ast$ closure of their supremands' sublevel sets; it can be viewed as an extension of~\cite[Proposition~7.1]{KrZ19}.

\begin{lemma}\label{lem:repres_relaxation_abstract}
Let $W:\R^m\times \R^m\to \R$ be symmetric, diagonal, lower semicontinuous, and coercive, and let $I: L^\infty(\Omega;\R^m)\to \R$ 
be a functional. Then, the following identities are equivalent:
\begin{itemize}
\item[$(i)$] $I=J_W^{\rm rlx};$\\[-0.3cm]
\item[$(ii)$] $L_c(I)=\Acal_{L_{c}(W)}^\infty$ for all $c\in \R$;  \\[-0.3cm]
\item[$(iii)$] $I(u) = \inf\{c\in \R: u\in \Acal_{L_c(W)}^\infty\}$. 
\end{itemize}
\end{lemma}
\begin{proof} 
The proof is straightforward, nevertheless, we detail the arguments here for the reader's convenience.\smallskip

$(ii)\Rightarrow (i)$
To prove that $I$ is a lower bound on the relaxation of $J_W$, consider $(u_j)_j\subset L^\infty(\Omega;\R^m)$ such that $u_j\weaklystar u$ in $L^\infty(\Omega;\R^m)$, assuming without loss of generality that 
\begin{align*}
\liminf_{j\to \infty}J_W(u_j) = \lim_{j\to \infty}J_W(u_j)=:c<\infty. 
\end{align*}
Fixing $\eps>0$, we find  $(u_j)_j\subset\Acal_{L_{c+\eps}(W)}$ for $j$ sufficiently large, and thus, $u\in \Acal^\infty_{L_{c+\eps}(W)}$, so that $I(u)\leq c+\eps = \liminf_{j\to \infty} J_W(u_j) + \eps$.
Since $(u_j)_j$ is an arbitrary weakly$^\ast$ converging sequence with limit $u$, taking the limit $\eps\to 0$ gives $I(u)\leq J_W^{\rm rlx}(u)$. 

For the reverse inequality, we observe that any $u\in L^\infty(\Omega;\R^m)$ satisfies $u\in L_{I(u)}(I)= \Acal_{L_{I(u)}(W)}^{\infty}$ and can therefore be approximated in the sense of $L^\infty$-weak$^\ast$ convergence by a sequence $(u_j)_j\subset \Acal_{L_{I(u)}(W)}$, so that
\begin{align*}
J^{\rm rlx}_W(u) \leq \limsup_{j\to \infty}J_W(u_j) =\limsup_{j\to \infty} \esssup_{(x,y)\in \Omega\times \Omega} W(u_j(x), u_j(y)) \leq I(u).
\end{align*}

$(i)\Rightarrow (ii)$ We show first that any $u\in L_c(I) = L_c(J_W^{\rm rlx})$ for $c\in \R$ is the weak$^\ast$ limit of a sequence in $\Acal_{L_c(W)}$. To that end, let $\eps>0$ and take a recovery sequence $(u_j)_j\subset L^\infty(\Omega;\R^m)$ for $u$. Then, $u_j\weaklystar u$ in $L^\infty(\Omega;\R^m)$ and $\lim_{j\to \infty} J_W(u_j)=J_W^{\rm rlx}(u)\leq c$. The latter implies $u_j\in\Acal_{L_{c+\eps}(W)}$ for all $j\in \N$ sufficiently large. Hence, $u\in \Acal_{L_{c+\eps}(W)}^\infty$,  and since $\eps>0$ was arbitrary, $u\in \Acal_{L_{c}(W)}^\infty$ by Lemma~\ref{lem:intersections_AKeps}. 

On the other hand, if $u\in \Acal_{L_c(W)}^\infty$ with $u_j\weaklystar u$ in $L^\infty(\Omega;\R^m)$ and $(u_j)_j\subset \Acal_{L_c(W)}$, then 
\begin{align*}
c\geq \liminf_{j\to \infty} J_W(u_j) \geq J_W^{\rm rlx}(u) = I(u),
\end{align*}
meaning $u\in L_c(I)$.\smallskip

 $(ii)\Leftrightarrow (iii)$ This follows from~\eqref{levelset-i} and~\eqref{level-sets-ii} while taking Lemma~\ref{lemma:basic-topo} into account. 
\end{proof}

By combining the statement of the previous lemma with the characterization of weak$^\ast$ closures of nonlocal inclusions in~\eqref{characterization_inclusionAcal1}, we infer that
\begin{align*}
J_W^{\rm rlx}(u)  & =  \inf\{c \in \R \colon u\in L^\infty(\Omega;B^{\rm co})\ \text{with $B\times B\in \Pcal_{ L_c(W)}$}\}
\end{align*}
for $u\in L^\infty(\Omega;\R^m)$. Even if this provides 
a new and simpler representation of the relaxation of $J_W$, it is not obvious from this formula
 whether (and - if not in general - under what assumptions) $J_W^{\rm rlx}$ corresponds again to a supremal functional of the form~\eqref{es}. We give a characterization result about when $J_W^{\rm rlx}$ is structure-preserving in the next section.

\subsection{Structure preservation under relaxation} 
As a consequence of Theorem~\ref{theo:charact-wls}, we can immediately deduce a lower bound on the relaxation of $J_W$ as in~\eqref{es} in the form of another nonlocal supremal of the same type, precisely,
\begin{align}\label{lb_JWrlx}
J_W^{\rm rlx}\geq J_{{\widehat{W}}^{\times \rm lc}},  
\end{align} 
where ${\widehat{W}}^{\times \rm lc}$ is the Cartesian level convex envelope of $\widehat{W}$, see Definition~\ref{def:wslc}. \color{black}
	
Our second main result of the paper, formulated in Theorem~\ref{theo:relaxation}, characterizes the conditions on $W$ under which equality holds in~\eqref{lb_JWrlx}. It is a direct implication of the following proposition, which again relies on the characterization of weak$^\ast$ closures of nonlocal inclusions from  Proposition~\ref{prop:structure_inclusions}.

\begin{proposition}\label{prop:structure-preservation}
	Let $W:\R^m\times \R^m\to \R$ be symmetric, diagonal, lower semicontinuous, and coercive.   \color{black}
Then, the following two statements are equivalent: 
\begin{itemize}
\item[$(i)$]  all sublevel sets of $W$ admit a basic Cartesian convexification;\\[-0.3cm]
\item[$(ii)$] there exists a symmetric, diagonal, lower semicontinuous, and coercive function $G:\R^m\times \R^m\to \R$ such that $J_W^{\rm rlx} = J_G$.
\end{itemize}
If these conditions are fulfilled, then $J_W^{\rm rlx}= J_{W^{\times \rm lc}}$.
\end{proposition}

\begin{proof}
  $(i)\Rightarrow (ii)$
 The Cartesian level convexification $W^{\times \rm lc}$ is symmetric, diagonal, lower semicontinuous, and coercive, and satisfies
\begin{align*}
L_c(W^{\times \rm lc}) = L_c(W)^{\times \rm c}\quad \text{ for all $c\in \R$,}
\end{align*} 
according to Lemma~\ref{lem:W-wslc}. 
Moreover, we infer  from Propositions~\ref{lem:repres_relaxation_abstract} and~\ref{prop:structure_inclusions} (see also~\eqref{levelset-i}) that
\begin{align*}
J_W^{\rm rlx}(u)& = \inf\{c\in\R: u\in \Acal_{L_c(W)}^\infty\} =  \inf\{c\in\R: u\in \Acal_{L_c(W)^{\times \rm c}}\}   \nonumber \\ 
&  =  \inf\{c\in\R: u\in \Acal_{L_c(W^{\times \rm lc})}\} = \inf\{c\in \R: u\in L_c(J_{W^{\times \rm lc}})\}  =  J_{W^{\times \rm lc}}(u) 
\end{align*} for $u\in L^\infty(\Omega;\R^m)$. 
Hence, setting $G=W^{\times \rm lc}$ proves $(ii)$. \smallskip

 $(ii)\Rightarrow (i)$ 
As $(ii)$ holds, we can use once again Proposition~\ref{lem:repres_relaxation_abstract} (and~\eqref{levelset-i}) to obtain
\begin{align*}
J_{G}(u)= J_W^{\rm rlx}(u) =\inf\{c\in\R: u\in \Acal_{L_c(W)}^\infty\} 
\end{align*}
and consequently, 
\begin{align*}
\Acal_{L_c(W)}^\infty = L_c(J_G)= \Acal_{L_c(G)} \quad \text{for every $c\in \R$}
\end{align*} 
by~\eqref{level-sets-ii} under consideration of Lemma~\ref{lem:intersections_AKeps}. Now, the claim follows  immediately from  Lemma~\ref{lem:aux1}. 
\end{proof}

Summing up, we have seen that, in general, $J_W^{\rm rlx}$ is not a supremal functional of the form~\eqref{es}. The next section discusses different (counter)examples of representation formulas in cases when $W$ are specific multi-well supremands.

\subsection{Relaxation of nonlocal supremals in multi-well form}\label{subsec:multiwell}

The intention of this section is to illustrate the previous results, using the example of 
$L^\infty$-functionals
$$
J_W(u):=\esssup_{(x, y)\in \Omega\times \Omega} W(u(x),u(y))\quad \text{for $u\in L^{\infty}(\Omega;\R^m)$ }
$$  
with multi-well supremands of the form
\begin{align*}
W=\min_{i=1, \ldots, N} W_i,
\end{align*}
where $W_i:\R^m\times \R^m\to \R$ is symmetric, diagonal, lower semicontinuous and coercive for $i=1, \ldots, N$ with $N\in \N$.

Below, we will often specify this setting further by choosing $W_i$ as single-well functions  
 \begin{align}\label{singlewellW}
 W_i(\xi, \zeta)=\dist\bigl((\xi, \zeta), A_i\times A_i\bigr) \quad \text{for $(\xi, \zeta)\in \R^m\times \R^m$}
 \end{align} 
 with compact sets $A_i\subset \R^m$ such that~\eqref{ass12} holds.
Moreover, for $(\xi, \zeta)\in \R^m\times \R^m$,
\begin{align}\label{W2}
 W(\xi, \zeta)=\dist\bigl((\xi, \zeta), K\bigr) \quad \text{ with $K:=\bigcup_{i=1}^N A_i\times A_i$. }
\end{align}
Observe for the sublevel sets that \begin{align*}
L_c(W) = \bigcup_{i=1}^N (A_i)_c\times (A_i)_c= \bigcup_{i=1}^N L_c(W_i) 
\end{align*}
with $c\in \R$ and $A_c:=A+B_c(0)$ for $A\subset \R^m$. 

Before we focus on relaxation formulas with and without structure preservation (see Example~\ref{ex:preservation} and~\ref{ex:counterexample}), let us provide the following basic representation for general multi-well $L^\infty$-functionals. 

\begin{lemma}\label{lem:min}
Let $W_i:\R^m\times \R^m\to \R$ for $i=1, \ldots, N$ with $N\in \N$ be symmetric, diagonal, lower semicontinuous, and coercive, and suppose that $W=\min_{i=1, \ldots, N} W_i$ satisfies 
\begin{equation}\label{assm}
\Pcal_{L_c(W)} \subset \bigcup_{i=1}^N\Pcal_{L_c(W_i)} \qquad \text{for all $c\in \R$}. 
\end{equation}
Then, 
\begin{align}\label{min44} 
J_W= \min_{i=1, \ldots, N} J_{W_i}.
\end{align}
\end{lemma}

\begin{proof}
Since the sublevel sets of $W$ are given by
 $L_c(W) = \bigcup_{i=1}^N L_c(W_i)$ for $c\in \R$, it is immediate to check that  the properties of the functions $W_i$ carry directly over to $W$, that is, also $W$ is symmetric, diagonal, lower semicontinuous, and coercive.
Considering~\eqref{levelset-i}, \eqref{AK-alt}, and~\eqref{assm}, this can be seen like this: For any $u\in L^\infty(\Omega;\R^m)$,
\begin{align*}
J_W(u) &=\inf\bigl\{c\in \R:u\in \Acal_{L_c(W)}\bigr\} =\textstyle \inf\bigl\{c\in \R: u\in \bigcup_{B\times B\in \Pcal_{L_c(W)}}L^\infty(\Omega;B)\bigr\} \\ &
\geq  \textstyle\inf\{c\in \R: u\in \bigcup_{B\times B\in \Pcal_{L_c(W_i)}} L^\infty(\Omega;B) \text{ for some $i\in \{1, \ldots, N\}$}\} \\ &
=  \min_{i=1, \ldots, N} \textstyle \inf\{c\in \R: u\in \bigcup_{B\times B\in \Pcal_{L_c(W_i)}} L^\infty(\Omega;B)\}\\ &
= \min_{i=1, \ldots, N} \inf\{c\in \R: u\in \Acal_{L_c(W_i)}\}  =  \min_{i=1, \ldots, N}J_{W_i}(u). 
\end{align*}
The reverse inequality is trivially satisfied due to $W\leq W_i$ for all $i$.
\end{proof}

Notice that the assumption~\eqref{assm} is necessary for proving~\eqref{min44} for $N>2$, as we show below for three-well supremands $W$. In particular, this highlights what implications the intrinsic non-uniqueness in the representation of $W$ as minimum of the functions $W_i$ has for the representation of $J_W$ and illustrates the role of hidden squares of the sublevel sets of $W$.

\begin{example}\label{ex:W=min} 
Let $W$ and $W_i$ for $i=1, \ldots, N$ be as in~\eqref{W2} and~\eqref{singlewellW}.\smallskip

a) In case $W$ is a double-well supremand, which corresponds to $N=2$, one always has $J_W=\min_{i=1,2}J_{W_i}$.
 Indeed, the condition~\eqref{assm} is satisfied here, as
	 Conclusion~\ref{rem:hidden-squares}\,b) implies  
	$$
	\Pcal_{L_c(W)} \subset \{(A_1)_c\times (A_1)_c\} \cup \{(A_2)_c\times (A_2)_c\} =  \Pcal_{L_c(W_1)} \cup  \Pcal_{L_c(W_2)}  
	$$
for all $c\in \R$.\smallskip

	b) If $N=3$, meaning that $W$ is a three-well supremand, it is in general not true that $J_W=\min_{i=1,2,3} J_{W_i}$. This observation is tied to the occurrence of hidden maximal Cartesian squares, which can affect the structure of nonlocal supremal functionals. Taking for instance compact sets $A_1, A_2, A_3 \subset \R^m$ as in Example~\ref{ex:hidden}, then $M\times M\in \Pcal_K=\Pcal_{L_0(W)}$ with $M$ defined in~\eqref{eqM}, but $M\times M\notin \Pcal_{L_0(W_i)} = \{A_i\times A_i\}$ for any $i=1,2,3$.

However, we will prove that it is possible to express $W$ as the minimum of four functions so that the condition~\eqref{assm} is satisfied. To this aim, consider the function $W_4:\R^m\times \R^m\to \R$ defined for $(\xi, \zeta)\in \R^m\times \R^m$ by
	\begin{equation}\label{W4}
	W_4(\xi,\zeta) =\inf \{c\in \R \colon (\xi,\zeta)\in M^c\times M^c \}
	\end{equation}
	with
\begin{equation}\label{Hc}
M^c := \bigcup_{i,j=1,2,3, i\neq j} (A_i)_c\cap (A_j)_c \quad \text{ for all } c\in \R.
\end{equation} 
Then, $L_c(W_4)= M^c\times M^c  \subset  L_c(W)=\bigcup_{i=1}^3 L_c(W_i)$ in view of~\eqref{level-sets-ii} implies $W\le W_4$ and thus,
	\begin{align}\label{4!}
		W=\min_{i=1, \ldots, 4} W_i.
	\end{align} 
By Lemma~\ref{prop:hidden-squares},  it is straightforward to verify that~\eqref{assm} is indeed fulfilled for this representation of $W$, 
and we conclude from Lemma~\ref{lem:min} that
	\begin{align*}
		J_W=\min_{i=1,2, 3,4}J_{W_i}.
	\end{align*}
	\vspace{-0.7cm}

	\end{example}

In view of Lemma~\ref{lem:min}, it is  not hard to see that the relaxation of $J_{W}$ with $W=\min_{i=1, \ldots, N} W_i$ satisfying~\eqref{assm} can be obtained via minimization over the relaxed supremals associated to the more basic functions $W_i$, that is,
\begin{align}\label{128}
J_W^{\rm rlx} = \min_{i=1, \ldots, N} J_{W_i}^{\rm rlx}.
\end{align} 

Indeed, as a consequence of~\eqref{min44}, we infer for the relaxation of $J_W$ that 
$$
J_{W_j}^{\rm rlx} \geq J_W^{\rm rlx} \ge \bigl(\min_{i=1, \dots, N} J_{W_i}^{\rm rlx}\bigr)^{\rm rlx}
$$
for all $j\in \{1, \ldots, N\}$. The identity~\eqref{128} follows then by taking the minimum over $j$, along with the observation that the minimum of finitely many weakly$^\ast$ lower semicontinuous functionals is again weakly$^\ast$ lower semicontinuous.

Without further hypotheses on $W_i$, it is not possible to deduce $J^{\rm rlx}_{W_i}= J_{W_i^{\times \rm lc}}$, as~Proposition~\ref{prop:structure-preservation} indicates. 
If $W_i$ is a single-well supremands like in~\eqref{singlewellW}, though, 
then $J_{W_i}$ is actually a classical $L^\infty$-functional, that is,
\begin{equation*}
J_{W_i}(u)  =\esssup_{(x,y)\in \Omega\times \Omega} \dist((u(x), u(y)), A_i\times A_i) = \esssup_{x\in \Omega} \dist(u(x), A_i),
\end{equation*}
cf.~\eqref{distmax}.
The relaxation of $J_{W_i}$ can therefore be determined via the classical theory of $L^\infty$-functionals  (see~\cite[Theorem~2.5]{Prinari06}), 
which gives
\begin{equation*}
J_{W_i}^{\rm rlx}(u)= \esssup_{x\in \Omega} \dist(u(x), A_i^{\rm co})
\end{equation*}
for $u\in L^\infty(\Omega;\R^m)$, or equivalently,
\begin{align*}
J^{\rm rlx}_{W_i}=J_{W_i^{\rm co}}=J_{W_i^{\times \rm lc}};
\end{align*} 
here $W_i^{\times \rm lc}(\xi, \zeta)= W_i^{\rm co}(\xi, \zeta)=\dist\bigl((\xi, \zeta), A^{\rm co}\times A^{\rm co}\bigr)$ for $(\xi, \zeta)\in \R^m\times \R^m$ and $W^{\rm co}$ denotes the convex envelope of $W$.

With these preparations, we now present two qualitatively different examples of relaxation results for nonlocal supremals of three-well type, which build on the insights from Examples~\ref{ex3} and~\ref{ex:W=min}. Even though, the relaxed functionals can both be calculated via the same formula through minimization over four nonlocal supremals (see~\eqref{99} below), one case is structure-preserving, while the other is not. 

\begin{example}[Structure preservation during relaxation]\label{ex:preservation} 

 Consider $W$ and $W_i$ for $i=1,2,3$ as~\eqref{W2} and~\eqref{singlewellW} with convex sets $A_1, A_2, A_3\subset \R^m$ with pairwise nonempty intersections. Then not only do the three sets $A_1, A_2, A_3$ fit into the framework of  Example~\ref{ex:W=min}\,b), but more generally, also $(A_i)_c=A_i+B_c(0)$ with $i=1,2,3$ for any $c\in \R$. Hence, all sublevel sets of $W$ have a basic Cartesian convexification and
\begin{align*}
L_c(W)^{\rm \times c}= \bigcup_{i=1}^3 [(A_i)_c \times  (A_i)_c] \cup [(M^c)^{\rm co}\times (M^c)^{\rm co}] =\bigcup_{i=1}^4 L_c(W_i)^{\rm \times c}
\end{align*}
for any $c\in \R$ with $M^c$ and $W_4$ as in~\eqref{Hc} and~\eqref{W4}, respectively.
This allows us to conclude with Proposition~\ref{prop:structure-preservation} that 
$$
J^{\rm rlx}_W = J_{W^{\times \rm lc}},
$$ 
where the the Cartesian level convex envelope of $W$ can be determined via $$W^{\times \rm lc} = \min_{i=1, 2,3,4} W_i^{\times \rm lc} = \min_{i=1,2,3,4} W_i^{\rm co}.$$ 
\end{example}
  
 \begin{example}[Loss of nonlocal supremal structure during relaxation]\label{ex:counterexample}
Let $W$ and $W_i$ for $i=1,2,3$ as~\eqref{W2} and~\eqref{singlewellW} with three sets $A_1, A_2, A_3\subset \R^m$ 
as in	 Example~\ref{ex3}\,b). Then $L_0(W)=K$ does not admit a basic Cartesian convexification, since
$$
\Pcal_{L_0(W)^{\times \rm c}} = \Pcal_{K^{\times \rm c}}=  \bigcup_{i=1}^3 \{A_i^{\rm co}\times A_i^{\rm co}\} \cup \{M^{\rm co}\times M^{\rm co}\} 
$$
but $M\times M \notin \Pcal_{L_0(W)}$. As a consequence of Proposition~\ref{prop:structure-preservation}, 
$$
J_W^{\rm rlx} \neq J_{W^{\times \rm lc}}.
$$
\vspace{-0.9cm}

\end{example}

We point out that in both examples above, $J_W^{\rm rlx}$ can be expressed through minimization over the relaxations of the more basic supremal functionals, namely,
\begin{align}\label{99}
J_W^{\rm rlx}= \min_{i=1,2,3,4} J_{W_i}^{\rm rlx} = \min_{i=1,2,3,4} J_{W_i^{\rm co}} = \min_{i=1,2,3,4} J_{W_i^{\times \rm lc}},
\end{align}
where $W_4$ is defined in~\eqref{W4}; this is a consequence of~\eqref{128}  and~\eqref{4!}. 
 And still,  only Example~\ref{ex:preservation} allows us to rewrite $J_W^{\rm rlx}$ as a nonlocal supremal functional.

\section{$L^p$-approximation}\label{sec:Lp-approx}

This section provides the proof of  the $\Gamma$-convergence result formulated in Theorem~\ref{theo:gral-vectorial-case}. Some of the arguments are inspired by~\cite[Theorem 3.1]{Champion-DePascale-Prinari}, where the authors show a power-law approximation result connecting single integrals and classical (local) supremal functionals.

\begin{proof}[Proof of Theorem \ref{theo:gral-vectorial-case}] 
We subdivide the proof into three natural steps, starting with the relative weak  sequential compactness in $L^q(\Omega;\R^m)$ of sequences of uniformly bounded energy and followed by the lower and upper bounds for the $\Gamma$-limits of $(I_W^p)_{p\geq q}$, see Definition~\ref{def:Gamma}.
Both bounds exploit the classical $L^p$-approximation of the $L^{\infty}$-norm as well as the relaxation result established in Section~\ref{sec:lower-rlx}. 

In the following, let $(p_j)_j \subset [q,\infty)$ be a sequence that converges indefinitely to $\infty$, i.e., $p_j\to \infty$ as $j\to \infty$.  Besides, suppose that $c_2=0$ in the growth condition~\eqref{lineargrowthandcoercivity}; otherwise, consider the translated double-integrand $W+c_2$ in place of $W$. 

	\medskip
	
\noindent	\textit{Step~1: Compactness.} 	
Let $(u_j)_j\subset L^{q}(\Omega;\R^m)$ with
\begin{align}\label{sup}
\sup_{j\in\N} I_W^{p_j}(u_j)
<\infty.
\end{align}
We show that $(u_j)_j$ is bounded in $L^q(\Omega;\R^m)$ which yields the existence of a subsequence of $(u_j)_j$ that converges weakly in $L^q(\Omega;\R^m)$. 

The growth condition on $W$ from below (see~\eqref{lineargrowthandcoercivity} with $c_2=0$) along with H\"older's inequality and $p_j\geq q$,
gives for each such $j$ that 
\begin{align}\label{est233}
I_W^{p_j}(u_j)& = \norm{W(v_{u_j})}_{L^{p_j}(\Omega\times \Omega)} \ge \Lcal^n(\Omega)^{2(\frac{1}{p_j}-\frac{1}{q})} \norm{W(v_{u_j})}_{L^{q}(\Omega\times \Omega)}  \\ 
& \ge c_1\Lcal^n(\Omega)^{2(\frac{1}{p_j}-\frac{1}{q})} \norm{v_{u_j}}_{L^{q}(\Omega\times \Omega;\R^m\times \R^m)}  \ge  c\, \Lcal^n(\Omega)^{\frac{2}{p_j}} \norm{u_j}_{L^{q}(\Omega;\R^m)},\nonumber
\end{align}	
where $v_{u_j}$ are the nonlocal fields as defined in~\eqref{eq:v_u} and $c>0$ is a constant independent of $j$.
 
In light of~\eqref{sup} and $\lim_{j\to \infty} \Lcal^{n}(\Omega)^{1/p_j}=1$, this shows that $(u_j)_j$ is uniformly bounded in $L^q(\Omega;\R^m)$, which implies immediately the existence of a weakly convergent subsequence of $(u_j)_j$ in $L^q(\Omega;\R^m)$.  

\medskip

\noindent \textit{Step~2: Lower bound.}  Let $(u_j)_j \subset L^q(\Omega;\R^m)$ and $u\in L^q(\Omega;\R^m)$ such that $u_j\rightharpoonup u$ in $L^q(\Omega;\R^m)$. 
	To show that
	\begin{equation*}
	\liminf_{j\to \infty} I_W^{p_j}(u_j) \ge I_W^\infty(u), 
	\end{equation*}
	we may assume without loss of generality that 
	$$
	\lim_{j\to \infty}I_W^{p_j}(u_j)=\liminf_{j\to \infty} I_W^{p_j}(u_j)<\infty,
	$$ 
	and thus, $u_j\in L^{p_j}(\Omega;\R^m)$ for all $j\in \N$.
	As a consequence of the estimate~\eqref{est233}, one has for any $r\ge q>1$ that
	\begin{align}\label{234}
 \lim_{j\to \infty}I_W^{p_j}(u_j) \geq  \Lcal^n(\Omega)^{-\frac{2}{r}}\liminf_{j\to\infty}I_W^{r}(u_j),
	\end{align} 
and $(u_j)_j$ can be regarded as uniformly bounded in $L^r(\Omega;\R^m)$.  
	 The latter allows us, in view of Proposition~\ref{prop:YM}\,$(i)$, to extract a subsequence (not relabeled) generating a Young measure $\nu \in L^{\infty}_{w}(\Omega; \mathcal{P}r(\R^m))$ with $u=[\nu]$, that is, 
	\begin{align*}
	u_j\stackrel{\rm YM}{\longrightarrow} \nu \quad \text{as $j\to \infty$.}
	\end{align*}
Then, by \cite[Proposition 2.3]{Ped97}, also the associated nonlocal fields $(v_{u_j})_j$ generate a Young measure, namely,
	\begin{align*}
	v_{u_j}\stackrel{\rm YM}{\longrightarrow} \Lambda \quad \text{as $j\to\infty$} \qquad \text{ with $\Lambda_{(x,y)}= \nu_x \otimes \nu_y$   for a.e. $(x,y)\in \Omega\times\Omega$, }
	\end{align*}
	and Proposition \ref{prop:YM}\,$(ii)$ implies 
\begin{align*}
\liminf_{j\to\infty}I_W^{r}(u_j) &= \liminf_{j\to \infty} \left( \int_{\Omega}\int_{\Omega} W^{r}(u_j(x), u_j(y)) \, d x\, dy  \right)^{\frac{1}{r}}\\
&\ge \left(\int_{\Omega}\int_{\Omega} \int_{\R^m\times \R^m} W^r(\xi, \zeta) \,d(\nu_x\otimes \nu_y)(\xi, \zeta) \, d x\, d y \right)^{\frac{1}{r}}.
\end{align*}
For any $r>s>1$, we can exploit the convexity of the map $t\mapsto t^{\frac{r}{s}}$ to obtain
$$
\liminf_{j\to\infty}I_W^{r}(u_j) \ge  \left(\int_{\Omega}\int_{\Omega} \left(\int_{\R^m\times \R^m} W^s(\xi, \zeta) \,d (\nu_x\otimes \nu_y)(\xi, \zeta) \right)^{\frac{r}{s}} \, d x\, d y \right)^{\frac{1}{r}}.
$$ 
Now, passing successively to the limits $r\to \infty$ and $s\to \infty$ leads 
via iterative classical approximation of the $L^{\infty}$-norm by~\eqref{Lpapprox-nu} (see also~\eqref{Lpapprox_classic}). 
to
\begin{align*}
\liminf_{r\to \infty}\liminf_{j\to\infty}I_W^{r}(u_j)  &\ge  \liminf_{s\to \infty} \esssup_{(x,y) \in \Omega \times \Omega}\left(\int_{\R^m\times\R^m} W^s(\xi, \zeta) \, d(\nu_x \otimes \nu_y)(\xi, \zeta)\right)^{\frac{1}{s}}\\[0.2cm]
& \geq  \displaystyle  \esssup_{(x,y) \in \Omega \times \Omega} [(\nu_x\otimes \nu_y)\text{-}\esssup_{(\xi, \zeta)\in \R^m\times\R^m}W(\xi, \zeta)]. 
\end{align*}

To sum up, we conclude together with~\eqref{234} and $\lim_{r\to \infty} \mathcal{L}^{n}(\Omega)^{\frac{2}{r}}=1$ that
	\begin{align}\label{eq887}
\lim_{j\to \infty}I_W^{p_j}(u_j) \geq	\liminf_{r\to \infty}\liminf_{j\to\infty}I_W^{r}(u_j)  \ge \esssup_{(x,y) \in \Omega \times \Omega} f_{W, \nu}(x, y) = \norm{f_{W, \nu}}_{L^\infty(\Omega\times \Omega)}, \nonumber 
	\end{align}
	where $f_{W, \nu}$ is defined as
\begin{equation*}
f_{W, \nu}(x,y):=(\nu_x\otimes \nu_y)\text{-}\esssup_{(\xi, \zeta)\in \R^m\times\R^m}W(\xi, \zeta) 
\end{equation*}
for $(x,y)\in \Omega\times\Omega$.
The proof of the liminf-inequality will be complete, once it is shown that
\begin{align}\label{287}
\norm{f_{W, \nu}}_{L^\infty(\Omega\times \Omega)} =  I_W^\infty(u).
\end{align}  

To this end, observe that the lower semicontinuity of $W$ implies 
\begin{align*}
f_{W, \nu}(x,y)=  \sup_{(\xi, \zeta)\in\supp(\nu_x\otimes \nu_y)} W(\xi, \zeta)\quad
\end{align*} 
for $(x,y)\in \Omega\times\Omega$  
(see~\eqref{ess-supp}), and therefore,
\begin{align*}
\norm{f_{W, \nu}}_{L^\infty(\Omega\times \Omega)} 
&= \inf\{c\in \R: f_{W, \nu}(x, y)\leq  c\text{ for a.e.~$(x, y)\in\Omega\times \Omega$}\} \\ 
&= \inf\{c\in \R: \supp(\nu_x\otimes\nu_y)\subset L_c(W) \text{ for a.e.~$(x, y)\in \Omega\times \Omega$}\} \nonumber\\
&=\inf\{c\in \R: \nu\in \Vcal_{L_c(W)}\}, \nonumber
\end{align*}
with $\Vcal_{L_c(W)}$ as introduced in~\eqref{Vcal_K}. 
In combination with the equivalence from~\eqref{equivalenceVcal}, 
stating that 
\begin{align*}
u=[\nu]\in \Acal_{L_c(W)}^\infty \qquad \text{if and only if}\qquad  \nu\in \Vcal_{L_c(W)}
\end{align*}
for any $c\in \R$, it follows that
\begin{align*}
\norm{f_{W, \nu}}_{L^\infty(\Omega\times \Omega)} = \inf\{c\in \R: u\in \Acal_{L_c(W)}^\infty\}=\begin{cases} J_{W}^{\rm rlx}(u) & \text{for $u\in L^\infty(\Omega)$,}\\ \infty &  \text{otherwise,}
\end{cases} 
\end{align*}
which is ~\eqref{287}. Note that the last identity uses Lemma~\ref{lem:repres_relaxation_abstract} and the simple fact that $u\notin \Acal^\infty_{L_c(W)}$ for any $c\in \R$ if $u\notin L^\infty(\Omega;\R^m)$.

\medskip	

	\noindent \textit{Step~3: Upper bound.} Let $u\in L^\infty(\Omega;\R^m)$. The construction of a recovery sequence  $(u_j)_j \subset L^q(\Omega;\R^m)$ with the properties that $u_j\rightharpoonup u$ in $L^q(\Omega;\R^m)$ as $j\to\infty$ and
	\begin{equation*}
	\limsup_{j\to \infty} I_W^{p_j}(u_j) \le  J_{W}^{\rm rlx}(u)
	\end{equation*}  
 is straightforward. Indeed, the definition of the relaxed functional in~\eqref{JW_rlx}
  provides a sequence $(u_k)_k \subset L^{\infty}(\Omega;\R^m)$ such that $u_k\rightharpoonup^* u$ in $L^{\infty}(\Omega;\R^m)$ as $k\to \infty$ and
	\begin{equation*}
	\lim_{k\to \infty} J_{W}(u_k)=J_{W}^{\rm rlx}(u). 
	\end{equation*} 
		Since $W(v_{u_k})$ is essentially bounded on $\Omega\times \Omega$ due to the growth assumption on $W$ in~\eqref{lineargrowthandcoercivity}, ~\eqref{Lpapprox-nu} and~\eqref{eq:J=Jhat} imply
	\begin{align*}
	\lim_{j\to \infty} I_W^{p_j}(u_k) &= \lim_{j\to \infty} \norm{W(v_{u_k})}_{L^{p_j}(\Omega\times \Omega)} 
= \norm{W(v_{u_k})}_{L^{\infty}(\Omega\times \Omega)} 
	 = J_W(u_k);
	\end{align*}	
	 recall the notation in~\eqref{eq:v_u}. 
	Hence, in the limit $k\to \infty$, 
	\begin{equation*}
	\lim_{k\to \infty} \lim_{j\to \infty} I_W^{p_j}(u_k) = J_{W}^{\rm rlx}(u),
	\end{equation*}
and we conclude by taking a diagonal sequence in the sense of Attouch, cf.~\cite[Corollary 1.16]{Att84}. 	
\end{proof}

\appendix
\section{Basic technical tools}

 We collect here a few auxiliary results that are useful for our analysis of Cartesian convex hulls. 

The following sequential compactness for sequences of compact sets with respect to the Hausdorff distance is a version of a classical statement (see e.g.~\cite[Section 3.9]{Gbook}) tailored to our needs.

\begin{lemma}\label{lem:Hausdorffconverg}
Let $(K_j)_j\subset \R^m\times \R^m$ be a sequence of compact sets with $K_{j+1}\subset K_j$ for all  $j\in \N$ and $K:=\bigcap_{j\in \N} K_j$. For any sequence of nonempty compact sets $(A_j)_j\subset \R^m$ such that $A_j\times A_j\subset K_j$ for all $j\in \N$, there exists a nonempty compact set $A\subset \R^m$ satisfying $A\times A\subset K$ and
\begin{align*}
\liminf_{j\to \infty} d_H(A_j^{\rm co}, A^{\rm co})=0.
\end{align*} 
\end{lemma}

\begin{proof}
 Since the projection onto the first vector variable of $K_1$, denoted by $\pi_1(K_1)$, is compact  and  $A_j\subset \pi_1(K_1)$ for all $j\in \N$, we exploit 
 the compactness of the metric space of nonempty compact subsets of a compact set in $\R^m$ equipped with the Hausdorff distance
  to find a compact set $A\subset \pi_1(K_1)$ as the Hausdorff limit of a (non-relabeled) subsequence of $(A_j)_j$, that is,
\begin{align}\label{55}
d_H(A_j, A)\to 0 \quad\text{as $j\to \infty$.}
\end{align}

A well-known result~\cite[Theorem 3.9.4]{Gbook} on the relation between Hausdorff distances and convexifications states that 
any two nonempty and compact sets $C, D\subset \R^m$
satisfy 
\begin{align*}
d_H(C^{\rm co}, D^{\rm co})\leq d_H(C, D).
\end{align*}
Applied to~\eqref{55}, this yields $d_H(A_j^{\rm co}, A^{\rm co})\to 0$ as $j\to \infty$.

To show that $A\times A\subset K$, one observes that for any $(\xi, \zeta)\in A\times A$ and $j\in \N$,
\begin{align*}
\dist\bigl((\xi, \zeta), K\bigr) &\leq \dist\bigl((\xi, \zeta), K_j\bigr) + d_H(K_j, K)\leq d_H(A\times A, A_j\times A_j)+d_H(K_j, K) \\ &\leq 
d_H(A_j, A)+d_H(K_j, K).
\end{align*}
Together with~\eqref{55} and $\lim_{j\to \infty}d_H(K_j, K)=0$, which follows from the fact that the limit in Hausdorff distance of a sequence of nested nonempty compact sets is exactly their intersection, we conclude that $(\xi, \zeta)\in K$.
\end{proof}

 The next lemma can be seen as the counterpart of \cite[Lemma 4.9]{KrZ19} for Cartesian convex hulls in place of separately convex ones.

\begin{lemma}\label{lemma:basic-topo} Let $(K_j)_j\subset\R^m\times \R^m$ be a sequence of nonempty, symmetric, diagonal, and compact sets with $K_{j+1}\subset K_j$ for all $j\in \N$. Then,  $\bigcap_{j\in\N} K_j^{\times \rm c} = (\bigcap_{j\in \N} K_j)^{\times \rm c}$. 
\end{lemma}

\begin{proof} Let  $K:=\bigcap_{j\in\N}K_j$.		
	While $K^{\times \rm c} = (\bigcap_{j\in\N} K_j)^{\times \rm c} \subset \bigcap_{j\in\N} (K_j)^{\times \rm c}$ is clear, the reverse inclusion follows based on the previous lemma and Lemma~\ref{lem:well-definition-Ewsc}. Indeed, we show below that
		\begin{align}\label{bigcap66}
	\bigcap_{j\in \N}(K_j)^{\times \rm c}_k \subset K^{\times \rm c}_k
	\end{align}
	for any $k\in \N$.
Then, Lemma~\ref{lem:well-definition-Ewsc} allows us to infer
	\begin{align*}
	\bigcap_{j\in \N} K_j^{\times \rm c} = \bigcap_{j\in \N} \bigcup_{k\in \N} (K_j)^{\times \rm c}_k  =  \bigcup_{k\in \N} \bigcap_{j\in \N}(K_j)^{\times \rm c}_k \subset  \bigcup_{k\in \N}  K_k^{\times \rm c} =K^{\times \rm c},
	\end{align*} 
	which yields the statement.

 We prove~\eqref{bigcap66} via induction. Starting with $k=1$, 
 let $(\xi, \zeta)\in (K_j)^{\times \rm c}_1$ for all $j\in \N$. 
Then there exists for each $j\in \N$ a set $B_j\subset\R^m$ with 
\begin{center}
$B_j\times B_j\in  \Pcal_{K_j}=\Pcal_{(K_j)^{\times \rm c}_0} $ \quad and \quad $\xi, \zeta\in B_j^{\rm co}$. 
\end{center}
Notice that all these $B_j$ are compact as maximal Cartesian squares of compact sets. 
According to Lemma~\ref{lem:Hausdorffconverg}, there is then a nonempty compact set $B\subset \R^m$ such that $B\times B\subset  K= \bigcap_{j\in \N}(K_j)_0^{\times \rm c}$ 
and
	\begin{align*}
\max\bigl\{ \dist(\xi, B^{\rm co}), \dist(\zeta, B^{\rm co})\bigr\}\leq \liminf_{j\to \infty} d_H(B_j^{\rm co}, B^{\rm co})
=0.
	\end{align*} 
	\color{black}
This shows $(\xi, \zeta)\in B^{\rm co}\times B^{\rm co}\subset K_{1}^{\times \rm c}$. 
Repeating the same argument successively for $k\in \N$, taking into account Lemma~\ref{lem:compactness}, 
yields~\eqref{bigcap66}. 
\end{proof}
\color{black}

For an arbitrary sequence of sets $(E_j)_j\subset \R^m\times \R^m$ it holds that 
$$
\bigcap_{j\in\N} \Acal_{E_{j}} =\Acal_{\cap_{j\in\N}E_j},
$$
see \cite[Lemma~5.2]{KrZ19}. 
We provide an asymptotic version of this identity when the sets are compact and nested.

\begin{lemma}\label{lem:intersections_AKeps}
Let $(K_j)_j\subset\R^m\times \R^m$ be a sequence of nonempty, symmetric, diagonal, and compact sets with $K_{j+1}\subset K_j$ for all $j\in \N$, and let $K:=\bigcap_{j\in\N}K_j$.
Then,
\begin{align*}
 \bigcap_{j\in\N} \Acal_{K_{j}}^\infty = \Acal_{K}^\infty. 
\end{align*}
\end{lemma}
\begin{proof}
To address the nontrivial inclusion, take $u\in  \bigcap_{j\in\N} \Acal_{K_{j}}^\infty$. In light of the representation formula~\eqref{characterization_inclusionAcal1},  there exists a sequence $(B_j)_j\subset \R^m$ with $B_j\times B_j\in \Pcal_{K_{j}}$ for all $j\in \N$ such that 
\begin{align*}
\textstyle u\in \bigcap_{j\in\N} L^\infty(\Omega;B_j^{\rm co}) = L^\infty(\Omega; \bigcap_{j\in\N} B_j^{\rm co}).
\end{align*} 
We can use Lemma~\ref{lem:Hausdorffconverg} 
 to infer the existence of a nonempty compact set $B\subset \R^m$ 
with $B\times B\subset K$ and 
\begin{align*}
\bigcap_{j\in \N}B_j^{\rm co} \subset B^{\rm co};
\end{align*} the details of this argument are similar to those in the proof of~\eqref{bigcap66}. \color{black}
Then, $u\in L^\infty(\Omega;B^{\rm co}) \subset \Acal_K^\infty$, exploiting again~\eqref{characterization_inclusionAcal1}. 
\end{proof}

%
%
\section*{Acknowledgments}
CK and AR were supported by the Dutch Research Council NWO through the project TOP2.17.012.
Part of this research was done while CK and AR were affiliated with Utrecht University; in particular, CK acknowledges partial support by a Westerdijk Fellowship from Utrecht University.
EZ is a member of INdAM GNAMPA whose support through GNAMPA Project 2020 \textquoteleft Analisi variazionale di modelli non-locali nelle scienze applicate\textquoteright \,is gratefully acknowledged.


\bibliographystyle{abbrv}
\bibliography{CKAREZ}

\begin{thebibliography}{10}

\bibitem{Acerbi-Buttazzo-Prinari}
E.~Acerbi, G.~Buttazzo, and F.~Prinari.
\newblock The class of functionals which can be represented by a supremum.
\newblock {\em J. Convex Anal.}, 9(1):225--236, 2002.

\bibitem{Acerbi-Fusco}
E.~Acerbi and N.~Fusco.
\newblock Semicontinuity problems in the calculus of variations.
\newblock {\em Arch. Rational Mech. Anal.}, 86(2):125--145, 1984.

\bibitem{Ansini-Prinari}
N.~Ansini and F.~Prinari.
\newblock Power-law approximation under differential constraints.
\newblock {\em SIAM J. Math. Anal.}, 46(2):1085--1115, 2014.

\bibitem{AnDiKh20}
H.~Antil, Z.~W. Di, and R.~Khatri.
\newblock Bilevel optimization, deep learning and fractional laplacian
  regularization with applications in tomography.
\newblock {\em Inverse Problems}, 36(6):064001, May 2020.

\bibitem{Ar65}
G.~Aronsson.
\newblock Minimization problems for the functional {${\rm
  sup}_{x}\,F(x,\,f(x),\,f^{\prime} (x))$}, {I-III}.
\newblock {\em Ark. Mat.}, 6:33--53 (1965), 1965.

\bibitem{Att84}
H.~Attouch.
\newblock {\em Variational convergence for functions and operators}.
\newblock Applicable Mathematics Series. Pitman (Advanced Publishing Program),
  Boston, MA, 1984.

\bibitem{Aumann-Hart}
R.~J. Aumann and S.~Hart.
\newblock Bi-convexity and bi-martingales.
\newblock {\em Israel J. Math.}, 54(2):159--180, 1986.

\bibitem{BaJeWa01}
E.~N. Barron, R.~R. Jensen, and C.~Y. Wang.
\newblock The {E}uler equation and absolute minimizers of {$L^\infty$}
  functionals.
\newblock {\em Arch. Ration. Mech. Anal.}, 157(4):255--283, 2001.

\bibitem{Barron-Jensen-Wang}
E.~N. Barron, R.~R. Jensen, and C.~Y. Wang.
\newblock Lower semicontinuity of {$L^\infty$} functionals.
\newblock {\em Ann. Inst. H. Poincar\'{e} Anal. Non Lin\'{e}aire},
  18(4):495--517, 2001.

\bibitem{BaLiu97}
E.~N. Barron and W.~Liu.
\newblock Calculus of variations in {$L^\infty$}.
\newblock {\em Appl. Math. Optim.}, 35(3):237--263, 1997.

\bibitem{BeCuMC2020}
J.~C. Bellido, J.~Cueto, and C.~Mora-Corral.
\newblock Bond-based peridynamics does not converge to hyperelasticity as the
  horizon goes to zero.
\newblock {\em J. Elasticity}, 141(2):273--289, 2020.

\bibitem{BeCuMC20}
J.~C. Bellido, J.~Cueto, and C.~Mora-Corral.
\newblock Fractional {P}iola identity and polyconvexity in fractional spaces.
\newblock {\em Ann. Inst. H. Poincar\'{e} Anal. Non Lin\'{e}aire},
  37(4):955--981, 2020.

\bibitem{BMC}
J.~C. Bellido and C.~Mora-Corral.
\newblock Lower semicontinuity and relaxation via {Y}oung measures for nonlocal
  variational problems and applications to peridynamics.
\newblock {\em SIAM J. Math. Anal.}, 50(1):779--809, 2018.

\bibitem{BePe06}
J.~Bevan and P.~Pedregal.
\newblock A necessary and sufficient condition for the weak lower
  semicontinuity of one-dimensional non-local variational integrals.
\newblock {\em Proc. Roy. Soc. Edinburgh Sect. A}, 136(4):701--708, 2006.

\bibitem{BrNg18}
H.~Brezis and H.-M. Nguyen.
\newblock Non-local functionals related to the total variation and connections
  with image processing.
\newblock {\em Ann. PDE}, 4(1):Paper No. 9, 77, 2018.

\bibitem{Briani-Garroni-Prinari}
A.~Briani, F.~Prinari, and A.~Garroni.
\newblock Homogenization of {$L^\infty$} functionals.
\newblock {\em Math. Models Methods Appl. Sci.}, 14(12):1761--1784, 2004.

\bibitem{ButDal83}
G.~Buttazzo and G.~Dal~Maso.
\newblock On {N}emyckii operators and integral representation of local
  functionals.
\newblock {\em Rend. Mat. (7)}, 3(3):491--509, 1983.

\bibitem{CarPiPri05}
P.~Cardaliaguet and F.~Prinari.
\newblock Supremal representation of {$L^\infty$} functionals.
\newblock {\em Appl. Math. Optim.}, 52(2):129--141, 2005.

\bibitem{Champion-DePascale-Prinari}
T.~Champion, L.~De~Pascale, and F.~Prinari.
\newblock {$\Gamma$}-convergence and absolute minimizers for supremal
  functionals.
\newblock {\em ESAIM Control Optim. Calc. Var.}, 10(1):14--27, 2004.

\bibitem{Dac08}
B.~Dacorogna.
\newblock {\em Direct methods in the calculus of variations}, volume~78 of {\em
  Applied Mathematical Sciences}.
\newblock Springer, New York, second edition, 2008.

\bibitem{DalMaso-book}
G.~Dal~Maso.
\newblock {\em An introduction to {$\Gamma$}-convergence}, volume~8 of {\em
  Progress in Nonlinear Differential Equations and their Applications}.
\newblock Birkh\"{a}user Boston, Inc., Boston, MA, 1993.

\bibitem{DMFoLe18}
G.~Dal~Maso, I.~Fonseca, and G.~Leoni.
\newblock Asymptotic analysis of second order nonlocal {C}ahn-{H}illiard-type
  functionals.
\newblock {\em Trans. Amer. Math. Soc.}, 370(4):2785--2823, 2018.

\bibitem{dEdLRTr20}
M.~D'Elia, J.~C.~D. los Reyes, and A.~M. Trujillo.
\newblock Bilevel parameter optimization for learning nonlocal image denoising
  models, 2020.

\bibitem{Dol03}
G.~Dolzmann.
\newblock {\em Variational methods for crystalline microstructure---analysis
  and computation}, volume 1803 of {\em Lecture Notes in Mathematics}.
\newblock Springer-Verlag, Berlin, 2003.

\bibitem{Eleuteri-Prinari}
M.~Eleuteri and F.~Prinari.
\newblock {$\Gamma$}-convergence for power-law functionals with variable
  exponents.
\newblock {\em Nonlinear Anal. Real World Appl.}, 58:103221, 21, 2021.

\bibitem{Fonseca-Leoni-book}
I.~Fonseca and G.~Leoni.
\newblock {\em Modern methods in the calculus of variations: {$L^p$} spaces}.
\newblock Springer Monographs in Mathematics. Springer, New York, 2007.

\bibitem{FoM99}
I.~Fonseca and S.~M\"{u}ller.
\newblock {$\mathcal{A}$}-quasiconvexity, lower semicontinuity, and {Y}oung
  measures.
\newblock {\em SIAM J. Math. Anal.}, 30(6):1355--1390, 1999.

\bibitem{Garroni-Nesi-Po}
A.~Garroni, V.~Nesi, and M.~Ponsiglione.
\newblock Dielectric breakdown: optimal bounds.
\newblock {\em R. Soc. Lond. Proc. Ser. A Math. Phys. Eng. Sci.},
  457(2014):2317--2335, 2001.

\bibitem{Gbook}
V.~V. Goncharov.
\newblock {\em An\`alise Mult\`ivoca e nao Suave}.
\newblock Universidade de \`Evora, 2018.

\bibitem{HoKu20}
G.~Holler and K.~Kunisch.
\newblock Learning nonlocal regularization operators, 2020.

\bibitem{KatBook}
N.~Katzourakis.
\newblock {\em An introduction to viscosity solutions for fully nonlinear {PDE}
  with applications to calculus of variations in {$L^\infty$}}.
\newblock SpringerBriefs in Mathematics. Springer, Cham, 2015.

\bibitem{Kolar}
J.~Kol\'{a}\v{r}.
\newblock Non-compact lamination convex hulls.
\newblock {\em Ann. Inst. H. Poincar\'{e} Anal. Non Lin\'{e}aire},
  20(3):391--403, 2003.

\bibitem{KrS22}
C.~Kreisbeck and H.~Sch\"{o}nberger.
\newblock Quasiconvexity in the fractional calculus of variations:
  characterization of lower semicontinuity and relaxation.
\newblock {\em Nonlinear Anal.}, 215:Paper No. 112625, 26, 2022.

\bibitem{KrZ19}
C.~Kreisbeck and E.~Zappale.
\newblock Lower semicontinuity and relaxation of nonlocal
  {$L^\infty$}-functionals.
\newblock {\em Calc. Var. Partial Differential Equations}, 59(4):Paper No. 138,
  36, 2020.

\bibitem{Kreisbeck-Zappale-2019Loss}
C.~Kreisbeck and E.~Zappale.
\newblock Loss of {D}ouble-{I}ntegral {C}haracter {D}uring {R}elaxation.
\newblock {\em SIAM J. Math. Anal.}, 53(1):351--385, 2021.

\bibitem{MenDu15}
T.~Mengesha and Q.~Du.
\newblock On the variational limit of a class of nonlocal functionals related
  to peridynamics.
\newblock {\em Nonlinearity}, 28(11):3999--4035, 2015.

\bibitem{TeMC20}
C.~Mora-Corral and A.~Tellini.
\newblock Relaxation of a scalar nonlocal variational problem with a
  double-well potential.
\newblock {\em Calc. Var. Partial Differential Equations}, 59(2):Paper No. 67,
  30, 2020.

\bibitem{Mor66}
C.~B. Morrey, Jr.
\newblock {\em Multiple integrals in the calculus of variations}.
\newblock Die Grundlehren der mathematischen Wissenschaften, Band 130.
  Springer-Verlag New York, Inc., New York, 1966.

\bibitem{Munoz}
J.~Mu\~{n}oz.
\newblock Characterisation of the weak lower semicontinuity for a type of
  nonlocal integral functional: the {$n$}-dimensional scalar case.
\newblock {\em J. Math. Anal. Appl.}, 360(2):495--502, 2009.

\bibitem{Ped97}
P.~Pedregal.
\newblock Nonlocal variational principles.
\newblock {\em Nonlinear Anal.}, 29(12):1379--1392, 1997.

\bibitem{Ped97-book}
P.~Pedregal.
\newblock {\em Parametrized measures and variational principles}, volume~30 of
  {\em Progress in Nonlinear Differential Equations and their Applications}.
\newblock Birkh\"{a}user Verlag, Basel, 1997.

\bibitem{Pedregal2016}
P.~Pedregal.
\newblock Weak lower semicontinuity and relaxation for a class of non-local
  functionals.
\newblock {\em Rev. Mat. Complut.}, 29(3):485--495, 2016.

\bibitem{Pe21}
P.~Pedregal.
\newblock On non-locality in the calculus of variations.
\newblock {\em SeMA J.}, 78(4):435--456, 2021.

\bibitem{Prinari06}
F.~Prinari.
\newblock Relaxation and {$\Gamma$}-convergence of supremal functionals.
\newblock {\em Boll. Unione Mat. Ital. Sez. B Artic. Ric. Mat. (8)},
  9(1):101--132, 2006.

\bibitem{Prinari09}
F.~Prinari.
\newblock Semicontinuity and relaxation of {$L^\infty$}-functionals.
\newblock {\em Adv. Calc. Var.}, 2(1):43--71, 2009.

\bibitem{Pri15}
F.~Prinari.
\newblock On the lower semicontinuity and approximation of
  {$L^\infty$}-functionals.
\newblock {\em NoDEA Nonlinear Differential Equations Appl.}, 22(6):1591--1605,
  2015.

\bibitem{PriZa20}
F.~Prinari and E.~Zappale.
\newblock A relaxation result in the vectorial setting and power law
  approximation for supremal functionals.
\newblock {\em J. Optim. Theory Appl.}, 186(2):412--452, 2020.

\bibitem{Rindler}
F.~Rindler.
\newblock {\em Calculus of variations}.
\newblock Universitext. Springer, Cham, 2018.

\bibitem{SaVa12}
O.~Savin and E.~Valdinoci.
\newblock {$\Gamma$}-convergence for nonlocal phase transitions.
\newblock {\em Ann. Inst. H. Poincar\'{e} Anal. Non Lin\'{e}aire},
  29(4):479--500, 2012.

\bibitem{Shieh-Spector}
T.-T. Shieh and D.~E. Spector.
\newblock On a new class of fractional partial differential equations.
\newblock {\em Adv. Calc. Var.}, 8(4):321--336, 2015.

\end{thebibliography}
\end{document}